\documentclass[11pt, a4paper]{article}

\usepackage{inputenc} 
\usepackage[english]{babel}
\usepackage{dsfont}
\usepackage{enumitem}
\setlist[itemize]{noitemsep}

\usepackage[totalheight=24 true cm, totalwidth=17 true cm]{geometry}

\usepackage{bbding}
\usepackage{amsmath,multirow}
\usepackage{amsthm}
\usepackage{amssymb}
\usepackage{subcaption}
\usepackage{mathrsfs}
\usepackage{mathtools}
\usepackage{stmaryrd}
\usepackage{url}
\usepackage{wrapfig}
\usepackage{starfont}
\usepackage{pifont}
\usepackage{eurosym}

\usepackage{imakeidx}
\makeindex[]

\newtheorem{theoremA}{Theorem}

\newtheorem{theoremB}{Theorem}

\newtheorem{theoremC}{Theorem}

\newcommand{\irow}[1]{
  \begin{smallmatrix}(#1)\end{smallmatrix}%
}
\newtheorem{procedure}{Procedure}

\usepackage{multicol}
\usepackage{relsize}
\usepackage{float}
\usepackage{tikz,pgf}
\usepackage{tikz-cd}
\usepackage{wasysym}
\usepackage{graphicx}
\usepackage{fancyhdr}
\usepackage{verbatim}

\usepackage{pgfplots}
\pgfplotsset{compat=1.15}
\usepackage{mathrsfs}
\usetikzlibrary{arrows}

\usepackage[all,dvips,knot,web,arc,curve,color,frame]{xy}
\usepackage[all,knot,arc,color,web]{xy}
\xyoption{arc}
\xyoption{web}
\xyoption{curve}

\numberwithin{equation}{section}

\theoremstyle{plain}
\newtheorem{theorem}{Theorem}[section]
\newtheorem{proposition}[theorem]{Proposition}
\newtheorem{corollary}[theorem]{Corollary}
\newtheorem{lemma}[theorem]{Lemma}
\newtheorem{notation}[theorem]{Notation}

\newtheorem{definition}[theorem]{Definition}

\theoremstyle{definition}

\newtheorem{remark}[theorem]{Remark}

\newtheorem{example}[theorem]{Example}

\usepackage[bookmarksnumbered=true]{hyperref} 
\hypersetup{
     colorlinks = true,
     linkcolor = blue,
     anchorcolor = blue,
     citecolor = teal,
     filecolor = blue,
     urlcolor = blue
     }
\newcommand\restr[2]{{
  \left.\kern-\nulldelimiterspace 
  #1 
  \vphantom{\small|} 
  \right|_{#2} 
  }}

\everymath{\displaystyle}
\allowdisplaybreaks

\makeatletter
\def\mathcenterto#1#2{\mathclap{\phantom{#1}\mathclap{#2}}\phantom{#1}}
\let\old@widetilde\widetilde
\def\widetildeto#1#2{\mathcenterto{#2}{\old@widetilde{\mathcenterto{#1}{#2\,}}}}
\let\old@widehat\widehat
\def\widehatto#1#2{\mathcenterto{#2}{\old@widehat{\mathcenterto{#1}{#2\,}}}}
\makeatother

\newcommand{\ip}[2]{\langle  #1,#2 \rangle} 
\newcommand{\is}[1]{\langle  #1 \rangle} 

\newcommand{\size}[1]{\left| #1 \right|} 
\newcommand{\pare}[1]{\left( #1 \right)} 
\newcommand{\set}[1]{{\left\{ #1 \right\}}} 

\newcommand{\corch}[1]{\left[ #1 \right]} 

\newcommand*\closure[1]{\overline{#1}}

\DeclareMathOperator{\rank}{rk}

\def\dim{\operatorname{dim}}

\DeclareMathOperator{\CC}{\mathbb{C}}

\DeclareMathOperator{\PP}{\mathbb{P}}

\DeclareMathOperator{\dnaive}{\textup{d}^{\textup{naive}}}

\hypersetup{
    colorlinks=true,
    linkcolor=blue,
    filecolor=magenta,      
    urlcolor=cyan,
    pdftitle={Overleaf Example},
    pdfpagemode=FullScreen,
    }
    
\def\Id{\operatorname{Id}}

\normalfont

\usepackage{bbm}
\usepackage{dsfont}

\usepackage{url}
\usepackage{amsmath,blkarray,booktabs, bigstrut}
\usepackage{tikz}
\usepackage{multirow}
\usepackage{xcolor}
\usepackage{color}
\usepackage[maxbibnames=99]{biblatex}

\renewbibmacro{in:}{}
\title{Irreducibility, Smoothness, and Connectivity of Realization Spaces of Matroids and Hyperplane Arrangements}
\author{Emiliano Liwski and Fatemeh Mohammadi}
\date{}

\begin{document}
\maketitle

\begin{abstract}
We study the realization spaces of matroids and hyperplane arrangements. First, we define the notion of naive dimension for the realization space of matroids and compare it with the expected dimension and the algebraic dimension, exploring the conditions under which these dimensions coincide. Next, we introduce the family of inductively connected matroids and investigate their realization spaces, establishing that they are smooth, irreducible, and isomorphic to a Zariski open subset of a complex space with a known dimension. Furthermore, we present an explicit procedure for computing their defining equations. As corollaries, we identify families of hyperplane arrangements whose moduli spaces are connected. Finally, we apply our results to study the rigidity of matroids. 
Rigidity, 
which involves matroids with a unique realization under projective transformations, is key to understanding the connectivity of the moduli spaces of the corresponding hyperplane arrangements. 
\end{abstract}

\section{Introduction}
This work focuses on the realization spaces of matroids and hyperplane arrangements, particularly on their geometric features, such as smoothness, irreducibility, and defining equations as algebraic sets, along with the topological properties of the complements of their associated hyperplane arrangements, including their combinatorial characterization and moduli space connectivity.

\smallskip
A matroid serves as a combinatorial framework for understanding linear dependence within a vector space; see e.g.,~\cite{Oxley}. When we have a finite collection of vectors in a specific vector space, the sets of linearly dependent vectors form a matroid. If we can reverse this process, meaning that for a given matroid $M$ we can identify a corresponding collection of vectors, we refer to these vectors as a realization of $M$. We denote the space of realizations of $M$ as $\Gamma_M$. The matroid variety $V_M$ represents the Zariski closure of the realization space, exhibiting a rich geometric structure. Matroid varieties were first introduced in \cite{gelfand1987combinatorial} and have since been extensively studied; see~e.g.,~\cite{sturmfels1989matroid,lee2013mnev,clarke2021matroid,sidman2021geometric}. 
\smallskip

Matroid varieties, and their associated ideals, have also been studied in the context of determinantal varieties in commutative algebra, \cite{bruns2003determinantal,clarke2020conditional,clarke2021matroid,herzog2010binomial,pfister2019primary,ene2013determinantal}, and they have applications in conditional independence models in algebraic statistics, \cite{Studeny05:Probabilistic_CI_structures,DrtonSturmfelsSullivant09:Algebraic_Statistics,Sullivant,clarke2021matroid,clarke2020conditional,clarke2022conditional,caines2022lattice, mohammadi2018prime}. 

\smallskip
One of the central questions regarding realization spaces of matroids is computing their irreducible decompositions and identifying families of matroids for which the associated varieties are irreducible. Another important question is determining their defining equations. Classical methods to tackle this problem include Grassmann-Cayley algebra~\cite{computationalgorithms,sidman2021geometric} and the geometric liftability technique~\cite{richter2011perspectives,Emiliano-Fatemeh4}.
In this paper, we introduce the family of inductively connected matroids and show the irreducibility of their realization spaces. We also outline a procedure for computing their defining polynomials. 

\smallskip
There is a natural correspondence between the realization spaces of matroids and hyperplane arrangements. Each vector collection realizing a matroid $M$ corresponds to a hyperplane arrangement in the dual vector space with a matching combinatorial structure. Thus, the matroid realization space serves as a parameter space for hyperplane arrangements with fixed combinatorics.

\smallskip

Complex hyperplane arrangements have long been central in algebraic combinatorics and algebraic topology; see~e.g.~\cite{orlik1980combinatorics,rybnikov2011fundamental,bartolo2004invariants,jiang1994diffeomorphic,wang2005rigidity,nazir2009admissible,garber2003pi1}. Let $\mathcal{H} = \{H_{1}, \ldots, H_{d}\}$ be a hyperplane arrangement in $\CC^{n}$, and let $N(\mathcal{H})$ denote its complement manifold. The arrangement $\mathcal{H}$ determines a matroid, denoted $M_{\mathcal{H}}$, often called the combinatorics of $M_{\mathcal{H}}$. This combinatorial structure encodes important topological information; for instance, the cohomology algebra of $N(\mathcal{H})$ is determined by $M_{\mathcal{H}}$. However, not all geometric aspects are fully captured by combinatorics alone. In \cite{rybnikov2011fundamental}, Rybnikov provided the first example of two arrangements 
with identical matroids but non-isomorphic fundamental groups.

\smallskip

The realization space of an arrangement $\mathcal{H}$ consists of all hyperplane arrangements sharing the same matroid as $\mathcal{H}$. As shown by Mn\"ev in \cite{mnev1988universality}, the geometry of these realization spaces can exhibit arbitrary complexity. More precisely, Mn\"ev's Universality Theorem asserts that for each singularity type, there exists a realization space exhibiting that singularity; see e.g.~\cite{mnev1985manifolds, mnev1988universality, sturmfels1989matroid,lee2013mnev}.

\smallskip

On the other hand, the topological features of realization spaces have also been extensively studied. 
For example, it was shown in \cite{randell1989lattice} that hyperplane arrangements within the same connected component of the realization space have diffeomorphic complements. This property holds particularly when the arrangements belong to the same irreducible component of the realization space. This motivates the study of the irreducibility of these spaces. As noted in \cite{knutson2013positroid}, this task is generally challenging, which motivates the investigation of specific families of arrangements.
For instance, \cite{nazir2012connectivity} established the irreducibility of realization spaces of inductively connected line arrangements. This work was recently extended in \cite{guerville2023connectivity}, showing that in such cases, the moduli space of the line arrangement is isomorphic to a Zariski open subset of a complex space with a known dimension. Furthermore, an explicit parametrization of the moduli space was provided.

\smallskip
In this paper, we define the family of inductively connected hyperplane arrangements and extend the results of \cite{nazir2012connectivity,guerville2023connectivity} to  arrangements of arbitrary rank.  Moreover, we provide a  characterization of rigid arrangements within this family and prove that they have a connected moduli space.

\smallskip

We now summarize the main results of this paper. In Section~\ref{sec 3}, we define the notion of naive dimension of the realization space of matroids, originally introduced in \cite{guerville2023connectivity} for line arrangements and rank 3 matroids. A related notion, the expected dimension, was provided in \cite{ford2015expected} as a recursive formula 
for computing the dimension of the matroid variety, the Zariski closure of the realization space \(\Gamma_{M}\). In particular, it was shown in \cite{ford2015expected} that, for the well-studied class of positroids (see \cite{gelfand1987combinatorial,postnikov2006total,knutson2013positroid,ardila2016positroids,mohammadi2024combinatorics, mohammadi2022computing}), this formula matches the actual dimension. We examine these three notions of dimensions: naive, expected, and actual dimensions of realization spaces, investigating cases in which they coincide. Our approach introduces a new perspective on these dimensions compared to previous methods.

\smallskip

We begin by defining the naive dimension, and refer to Definition~\ref{exp dimension} for the expected dimension.

\vspace{-1mm}

\begin{definition}[Definition~\ref{naiv defi}]\normalfont
Let $M$ be a matroid of rank $n$ with $d$ elements. 
Let $\mathcal{L}_M$ denote the set of subspaces of $M$ as in Definition~\ref{general}. The \textit{naive dimension} of the realization space $\Gamma_{M}$ of $M$ is: 
\[\dnaive(\Gamma_{M})=nd-\sum_{l\in \mathcal{L}_M}(\size{l}-\rank(l))(n-\rank(l)).
\]
\end{definition}
\vspace{-2mm}

We compute $\dnaive(\Gamma_{M})$ for various families of matroids and compare it with $\dim(\Gamma_{M})$ as follows. 

\begin{theoremA}
Let $M$ be a matroid of rank $n$ on the ground set $[d]$. The following statements hold:
\begin{itemize}
\item
If $M$ is realizable, then
$\dnaive(\Gamma_{M})\leq \dim(\Gamma_{M})$. \hfill{\textup{(Theorem~\ref{bound on dimension})}}

\item
If $M$ is an elementary split matroid of a hypergraph $\mathcal{H}=\{H_{1},\ldots,H_{q}\}$, then 
\[\dnaive(\Gamma_{M})=nd-\text{ec}(M)=nd-\sum_{i=1}^{q}(n-r_{i})(\size{H_{i}}-r_{i})\tag{Theorem~\ref{teo coincid}}.\]
\item
If $M$ is elementary split and inductively connected, then 
\[\dnaive( \Gamma_{M})=\dim(\Gamma_{M})\tag{Corollary~\ref{gen 2}}.\]
\item
If $M$ is elementary split, inductively connected, and rigid, then
\[\dnaive(\Gamma_{M})=\dim(\Gamma_{M})=n^{2}-1+d\tag{Proposition~\ref{propo rig}}.\]
\end{itemize}
\end{theoremA}

In Definition~\ref{def simp}, we introduce the class of {\em inductively connected} matroids, previously studied in the context of line arrangements or, equivalently, rank 3 matroids in \cite{nazir2012connectivity}. We extend this property for arbitrary rank matroids and study their realization spaces and matroid varieties as follows.

\begin{theoremB}
Let $M$ be an inductively connected matroid. The following statements hold:
\begin{itemize}
\item
$\Gamma_{M}$ is isomorphic to a Zariski open subset of a complex space. \hfill 
\textup{(Theorem~\ref{gen})}
\smallskip
\item
$\Gamma_{M}$ is smooth and irreducible. \hfill
\textup{(Corollary~\ref{coro pa})}
\smallskip

\item
If two hyperplane arrangements $\mathcal{H}_{1}$ and  $\mathcal{H}_{2}$ have $M$ as their matroids, then 
\[N(\mathcal{H}_{1})\cong N(\mathcal{H}_{2}),\]
where $N(\mathcal{H}_{1})$ and $N(\mathcal{H}_{2})$ denote the corresponding complement manifolds.\hfill \textup{(Corollary~\ref{coro pa})}
\end{itemize}
\end{theoremB}
Additionally, in Procedure~\ref{procedure}, we provide an explicit construction of the defining polynomials of the aforementioned Zariski open subset of $\Gamma_M$ and provide detailed examples in \S\ref{subsec:ex}.

Finally, we apply our results to investigate the rigidity of matroids and the connectivity of the moduli space of their corresponding hyperplane arrangements. The concept of rigidity, which pertains to matroids with a unique realization up to projective transformations, is central to the study of the connectivity of moduli spaces of hyperplane arrangements (see, for instance, \cite{guerville2023connectivity}). Furthermore, projective uniqueness has been explored in various other contexts (see \cite{brandt2019slack, ziegler1990matroid, miller2003unique, gouveia2020projectively, wenzel1991projective}). In \S\ref{subsec:rigid}, we utilize our main results to provide a complete characterization of the rigid matroids that are both inductively connected and elementary split. More precisely, we prove the following theorem.

\begin{theoremC}
Let $M$ be a matroid of rank $n$ on the ground set $[d]$ that is both inductively connected and elementary split. Then the following statements hold:
\begin{itemize}
\item
$M$ is rigid if and only if
$\dnaive(\Gamma_{M})=n^{2}-1+d.$ \hfill{\textup{(Proposition~\ref{propo rig})}}
\item
If $M$ is realizable, then $M$ is inductively rigid if and only if there exists an ordering $(p_{1},\ldots,p_{d})$ of its elements satisfying:
\begin{itemize}
\item $\{p_{1},\ldots,p_{n+1}\}$ forms a circuit.

\item $\tau_{i}=1$ for all $i\geq n+2$, where $\tau_{i}$ is given as in Definition~\ref{defi f}.\hfill{\textup{(Proposition~\ref{prop loc})}}

\end{itemize}
\end{itemize}
\end{theoremC}

\vspace{-2mm}

\medskip\noindent{\bf Outline.}
Section~\ref{sec 2} presents key concepts, such as matroid realization spaces and elementary split matroids. In Section~\ref{sec 3}, we introduce the notion of {\em naive dimension} of the realization space, proving it serves as a lower bound for the actual dimension and coincides with the expected dimension for the class of elementary split matroids.
Section~\ref{sec 4} introduces the family of {\em inductively connected} matroids, establishing the irreducibility and smoothness of their realization spaces, along with an explicit formula for their dimension. Additionally, we provide a characterization of rigid matroids of this type.


\section{Preliminaries}\label{sec 2}

\begin{notation}\label{notacion inicial}
We establish the following notation, which will be used throughout this note: 
\begin{itemize}
\item We denote the set $\{1,\ldots,n\}$ by $[n]$, and the collection of all subsets of $[d]$ of size $n$ by $\textstyle \binom{[d]}{n}$. 
\item For an $n\times d$ matrix $X=\pare{x_{i,j}}$ of indeterminates, let $\CC[X]$ denote the polynomial ring in the variables $x_{ij}$. For subsets $A\subset[n]$ and $B\subset[d]$ with $\size{A}=\size{B}$, we write  $[A|B]_{X}$ for the minor of $X$ with rows indexed by $A$ and columns indexed by $B$. 
\item The symbol $M$ will always refer to a rank $n$ matroid on a ground set of size $d$. 
\end{itemize}
\end{notation}


\subsection{Matroids}

We begin by reviewing key notions related to matroids; see \cite{Oxley} for further details.

\begin{definition}\normalfont
A matroid $M$ consists of a ground set $[d]$ together with a collection $\mathcal{C}$ of subsets of $[d]$, called
circuits, that satisfy the following three axioms: 
\begin{itemize}
\item $\emptyset \not\in\mathcal{C}$.
\item If $C_1,C_2 \in \mathcal{C}$ and $C_1\subseteq C_2$, then $C_1=C_2$.
\item  
 If $e\in C_1\cap C_2$ with $C_{1}\neq C_{2}\in \mathcal{C}$, there exists $C_3\in\mathcal{C}$ such that $C_3\subseteq (C_1\cup C_2)\backslash e$. 
\end{itemize}
The set of circuits of $M$ is denoted by $\mathcal{C}(M)$. Additionally, we define the following notions: 
\begin{itemize}
\item A subset of $[d]$ that contains a circuit is called {\em dependent}, while any subset that does not contain a circuit is {\em independent}. The collection of all dependent sets of $M$ is denoted by $\mathcal{D}(M)$. 
\item A {\em basis} is a maximal independent subset of $[d]$, with respect to inclusion. The set of all bases is denoted by $\mathcal{B}(M)$.
\item For any $F\subset [d]$, the {\em rank} of $F$, denoted $\rank_M(F)$ or simply $\rank(F)$, is the size of the largest independent set contained in $F$. We define $\rank(M)$  as $\rank([d])$, which is the size of any basis of $M$. 
\item For a subset $F\subset [d]$, we say that $x\in [d]$ belongs to the {\em closure} of $F$, denoted $x\in \closure{F}$, if $\rank(F\cup \{x\})=\rank(F)$. A set is called a {\em flat} if it is equal to its own closure.
\end{itemize}
\end{definition}

We now recall the matroid operations of deletion and contraction.

\begin{definition}
\normalfont \label{subm}
\begin{itemize}
\item For any subset $S \subseteq [d]$, the {\em restriction} of $M$ to $S$ is the matroid 
 on $S$ whose rank function is the restriction of the rank function of $M$ to $S$.
Unless otherwise specified, we assume that subsets of $[d]$ possess this structure and refer to them as \textit{submatroids} of $M$. This submatroid is denoted by $M|S$, or simply $S$ when the context is clear. The {\em deletion} of $S$, is denoted by $M\backslash S$, which corresponds to
$M|([d]\backslash S)$.
\item If $S$ is an independent subset of $M$, the {\em contraction} of $M$ by $S$ is the matroid $M/S$ on $[d]\backslash S$ where the rank of any subset $T$ is given by $\rank_{M}(T\cup S)-\rank_{M}(S)$.
\end{itemize}
\end{definition}

We now introduce the concept of {\em subspaces} within a matroid.

\begin{definition}\normalfont\label{general}
Let $M$ be a matroid of rank $n$, with elements, called points, denoted by $\mathcal{P}_{M}$. We define an equivalence relation on the circuits of $M$ of size less than $n+1$: 
\begin{equation}\label{equiv}C_{1}\sim C_{2} \Longleftrightarrow \closure{C_{1}}=\closure{C_{2}}.\end{equation}
We fix the following notation:
\begin{itemize}
\item An equivalence class $l$ is referred to as a {\em subspace} of $M$. We say that $\rank(l)=k$ if $\rank(C)=k$ for any circuit $C\in l$. We denote by $\mathcal{L}_{M}$ the set of all subspaces of the matroid $M$.
\item A point $p\in \mathcal{P}_{M}$ is said to belong to the subspace $l$ if there exists a circuit $C\in l$ with $p\in C$. For each $p\in \mathcal{P}_{M}$, let $\mathcal{L}_{p}$ denote the set of all the subspaces of $M$ containing $p$. 
The {\em degree} of $p$, denoted by deg$(p)$ is defined as $\size{\mathcal{L}_{p}}$. 
\item 
Let $\gamma$ be a collection of vectors of $\CC^{n}$ indexed by $\mathcal{P}_{M}$. We denote as $\gamma_{p}$ the vector of $\gamma$ corresponding to $p$. For a subspace $l\in \mathcal{L}_{M}$, we denote by $\gamma_{l}$  the subspace $\is{\gamma_{p}:p\in l}\subset \CC^{n}$. 
\end{itemize}
For simplicity, when the context is clear, we drop the index $M$ and 
write $\mathcal{L}_M$ and $\mathcal{P}$ instead. 
\end{definition}

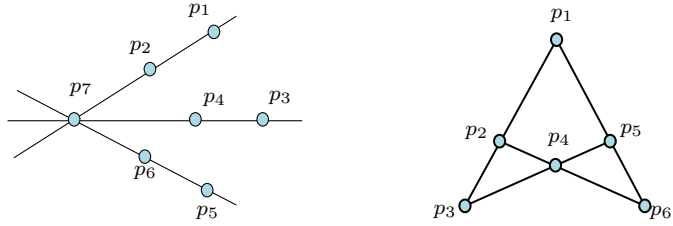
\begin{figure}[h]
    \centering
    \begin{subfigure}[b]{0.3\textwidth}
        \centering
        \begin{tikzpicture}[x=0.75pt,y=0.75pt,yscale=-1,xscale=1]

\tikzset{every picture/.style={line width=0.75pt}} 

\draw    (81.69,116.61) -- (191.16,174.3) ;
\draw    (77,131.88) -- (224,131.88) ;
\draw    (80.13,150.55) -- (191.16,79.27) ;
\draw [fill={rgb, 255:red, 173; green, 216; blue, 230}, fill opacity=1]
(107.34,131.2) .. controls (107.34,133.07) and (108.63,134.58) .. (110.23,134.58) .. controls (111.83,134.58) and (113.12,133.07) .. (113.12,131.2) .. controls (113.12,129.33) and (111.83,127.81) .. (110.23,127.81) .. controls (108.63,127.81) and (107.34,129.33) .. (107.34,131.2) -- cycle ;
\draw [fill={rgb, 255:red, 173; green, 216; blue, 230}, fill opacity=1]
(142.62,149.93) .. controls (142.62,151.8) and (143.91,153.32) .. (145.51,153.32) .. controls (147.11,153.32) and (148.4,151.8) .. (148.4,149.93) .. controls (148.4,148.06) and (147.11,146.55) .. (145.51,146.55) .. controls (143.91,146.55) and (142.62,148.06) .. (142.62,149.93) -- cycle ;
\draw [fill={rgb, 255:red, 173; green, 216; blue, 230}, fill opacity=1]
(173.7,166.79) .. controls (173.7,168.66) and (174.99,170.18) .. (176.59,170.18) .. controls (178.19,170.18) and (179.48,168.66) .. (179.48,166.79) .. controls (179.48,164.92) and (178.19,163.41) .. (176.59,163.41) .. controls (174.99,163.41) and (173.7,164.92) .. (173.7,166.79) -- cycle ;
\draw [fill={rgb, 255:red, 173; green, 216; blue, 230}, fill opacity=1] (201.42,131.2) .. controls (201.42,133.07) and (202.71,134.58) .. (204.31,134.58) .. controls (205.91,134.58) and (207.2,133.07) .. (207.2,131.2) .. controls (207.2,129.33) and (205.91,127.81) .. (204.31,127.81) .. controls (202.71,127.81) and (201.42,129.33) .. (201.42,131.2) -- cycle ;
\draw [fill={rgb, 255:red, 173; green, 216; blue, 230}, fill opacity=1] (167.82,131.2) .. controls (167.82,133.07) and (169.11,134.58) .. (170.71,134.58) .. controls (172.31,134.58) and (173.6,133.07) .. (173.6,131.2) .. controls (173.6,129.33) and (172.31,127.81) .. (170.71,127.81) .. controls (169.11,127.81) and (167.82,129.33) .. (167.82,131.2) -- cycle ;
\draw [fill={rgb, 255:red, 173; green, 216; blue, 230}, fill opacity=1] (177.06,87.18) .. controls (177.06,89.05) and (178.35,90.56) .. (179.95,90.56) .. controls (181.55,90.56) and (182.84,89.05) .. (182.84,87.18) .. controls (182.84,85.31) and (181.55,83.79) .. (179.95,83.79) .. controls (178.35,83.79) and (177.06,85.31) .. (177.06,87.18) -- cycle ;
\draw [fill={rgb, 255:red, 173; green, 216; blue, 230}, fill opacity=1] (145.14,105.91) .. controls (145.14,107.78) and (146.43,109.3) .. (148.03,109.3) .. controls (149.63,109.3) and (150.92,107.78) .. (150.92,105.91) .. controls (150.92,104.04) and (149.63,102.52) .. (148.03,102.52) .. controls (146.43,102.52) and (145.14,104.04) .. (145.14,105.91) -- cycle ;

\draw (166.25,71.57) node [anchor=north west][inner sep=0.75pt]   [align=left] {{\scriptsize $p_1$}};
\draw (206.28,115.51) node [anchor=north west][inner sep=0.75pt]   [align=left] {{\scriptsize $p_{3}$}};
\draw (138.91,153.35) node [anchor=north west][inner sep=0.75pt]   [align=left] {{\scriptsize $p_6$}};
\draw (170.08,173.62) node [anchor=north west][inner sep=0.75pt]   [align=left] {{\scriptsize $p_5$}};
\draw (173.44,116.45) node [anchor=north west][inner sep=0.75pt]   [align=left] {{\scriptsize $p_4$}};
\draw (106.8,111.36) node [anchor=north west][inner sep=0.75pt]   [align=left] {{\scriptsize $p_7$}};
\draw (136.66,90.59) node [anchor=north west][inner sep=0.75pt]   [align=left] {{\scriptsize $p_2$}};

   \end{tikzpicture}
        \label{fig:quadrilateral 2}
    \end{subfigure}
    \begin{subfigure}[b]{0.3\textwidth}
        \centering

\tikzset{every picture/.style={line width=0.75pt}} 

\begin{tikzpicture}[x=0.75pt,y=0.75pt,yscale=-1,xscale=1]

\draw [line width=0.75]    (247.65,123.92+60) -- (201.96,207.53+60) ;
\draw [line width=0.75]    (247.65,123.92+60) -- (291.68,207.53+60) ;
\draw [line width=0.75]    (219.23,174.98+60) -- (291.68,207.53+60) ;
\draw [line width=0.75]    (274.41,174.98+60) -- (201.96,207.53+60) ;
\draw [fill={rgb, 255:red, 173; green, 216; blue, 230}, fill opacity=1] (244.87,123.92+60) .. controls (244.87,125.68+60) and (246.12,127.11+60) .. (247.65,127.11+60) .. controls (249.19,127.11+60) and (250.44,125.68+60) .. (250.44,123.92+60) .. controls (250.44,122.15+60) and (249.19,120.73+60) .. (247.65,120.73+60) .. controls (246.12,120.73+60) and (244.87,122.15+60) .. (244.87,123.92+60) -- cycle ;
\draw [fill={rgb, 255:red, 173; green, 216; blue, 230}, fill opacity=1] (244.31,187.1+60) .. controls (244.31,188.87+60) and (245.56,190.29+60) .. (247.1,190.29+60) .. controls (248.64,190.29+60) and (249.88,188.87+60) .. (249.88,187.1+60) .. controls (249.88,185.34+60) and (248.64,183.91+60) .. (247.1,183.91+60) .. controls (245.56,183.91+60) and (244.31,185.34+60) .. (244.31,187.1+60) -- cycle ;
\draw [fill={rgb, 255:red, 173; green, 216; blue, 230}, fill opacity=1] (216.45,174.98+60) .. controls (216.45,176.74+60) and (217.69,178.17+60) .. (219.23,178.17+60) .. controls (220.77,178.17+60) and (222.02,176.74+60) .. (222.02,174.98+60) .. controls (222.02,173.21+60) and (220.77,171.79+60) .. (219.23,171.79+60) .. controls (217.69,171.79+60) and (216.45,173.21+60) .. (216.45,174.98+60) -- cycle ;
\draw [fill={rgb, 255:red, 173; green, 216; blue, 230}, fill opacity=1] (271.62,174.98+60) .. controls (271.62,176.74+60) and (272.87,178.17+60) .. (274.41,178.17+60) .. controls (275.94,178.17+60) and (277.19,176.74+60) .. (277.19,174.98+60) .. controls (277.19,173.21+60) and (275.94,171.79+60) .. (274.41,171.79+60) .. controls (272.87,171.79+60) and (271.62,173.21+60) .. (271.62,174.98+60) -- cycle ;
\draw [fill={rgb, 255:red, 173; green, 216; blue, 230}, fill opacity=1] (288.9,207.53+60) .. controls (288.9,209.29+60) and (290.14,210.72+60) .. (291.68,210.72+60) .. controls (293.22,210.72+60) and (294.47,209.29+60) .. (294.47,207.53+60) .. controls (294.47,205.77+60) and (293.22,204.34+60) .. (291.68,204.34+60) .. controls (290.14,204.34+60) and (288.9,205.77+60) .. (288.9,207.53+60) -- cycle ;
\draw [fill={rgb, 255:red, 173; green, 216; blue, 230}, fill opacity=1] (199.17,207.53+60) .. controls (199.17,209.29+60) and (200.42,210.72+60) .. (201.96,210.72+60) .. controls (203.49,210.72+60) and (204.74,209.29+60) .. (204.74,207.53+60) .. controls (204.74,205.77+60) and (203.49,204.34+60) .. (201.96,204.34+60) .. controls (200.42,204.34+60) and (199.17,205.77+60) .. (199.17,207.53+60) -- cycle ;

\draw (243.29,107.07+60) node [anchor=north west][inner sep=0.75pt]   [align=left] {{\scriptsize $p_1$}};
\draw (200.77,165.67+60) node [anchor=north west][inner sep=0.75pt]   [align=left] {{\scriptsize $p_2$}};
\draw (184.94,205.77+60) node [anchor=north west][inner sep=0.75pt]   [align=left] {{\scriptsize $p_3$}};
\draw (241.88,169.35+60) node [anchor=north west][inner sep=0.75pt]   [align=left] {{\scriptsize $p_4$}};
\draw (277.98,164.77+60) node [anchor=north west][inner sep=0.75pt]   [align=left] {{\scriptsize $p_5$}};
\draw (292.84,206.21+60) node [anchor=north west][inner sep=0.75pt]   [align=left] {{\scriptsize $p_6$}};

\end{tikzpicture}
        \label{fig:pascal}
    \end{subfigure}
\caption{(Left) Three concurrent lines; (Right) Quadrilateral set.}
    \label{fig:combined}
\end{figure}  

\begin{example}
Consider the quadrilateral set configuration $\text{QS}$ depicted in Figure~\ref{fig:combined} (Right), which represents a rank 3 matroid with points $\mathcal{P}_{\text{QS}}=\{p_{1},\ldots,p_{6}\}$, and the 
circuits (of size at most three): 
\[\mathcal{C}=\{\{p_1,p_2,p_3\},\{p_1,p_5,p_{6}\},\{p_3,p_4,p_5\},\{p_2,p_4,p_6\}\}.\]
The set of subspaces of $\text{QS}$ is equal to $\mathcal{C}$, and all points have degree two, i.e., $\deg(p_i)=2$ for all $i$. 
\end{example}

\subsection{Elementary split matroids}\label{sec:split}

The elementary split matroids were introduced in \cite{berczi2023hypergraph} through a hypergraph characterization and form a proper subclass of split matroids, which were originally defined in \cite{joswig2017matroids} using polyhedral geometry. 

\medskip
We recall the following characterization of elementary split matroids from 
\cite[Theorem~6]{berczi2023hypergraph}.

\begin{definition}
\normalfont 
\label{elementary}
Let $d\geq n$ be positive integers, and let $\mathcal{H}=\{H_{1},\ldots,H_{q}\}$ be a collection of subsets of $[d]$. Let $r_{1},\ldots,r_{q}$ be positive integers satisfying the following conditions for all $i,j\in [q]$: 
\[\size{H_{i}\cap H_{j}}\leq r_{i}+r_{j}-n,\quad 
  \size{[d]\backslash H_{i}}\geq n-r_{i},\quad 
  r_{i}\leq n-1, \quad 
  \size{H_{i}}\geq r_{i}+1.
\]
Then, the collection \[\mathcal{I}=\{X\subset [d]:\size{X}\leq n,\size{X\cap H_{i}}\leq r_{i}\ \text{for all}\ i\in [q]\}\] forms the independent sets of a matroid $M$ of rank $n$ on $[d]$ with rank function: 
\[\rank_{M}(F)=\text{min}\{n,\size{F},\text{min}_{i\leq q}\{\size{F-H_{i}}+r_{i}\}\}.\]
The matroid constructed in this manner is referred to as the {\em elementary split matroid} of $\mathcal{H}$.
\end{definition}

A notable subfamily of elementary split matroids is the family of paving matroids. It is conjectured in \cite{mayhew2011asymptotic} that asymptotically almost all matroids are paving; see also \cite{lowrance2013properties, mayhew2011asymptotic, oxley1991ternary, brualdi1972foundations, Oxley, welsh2010matroid}.

\begin{definition}[Paving matroid]\normalfont \label{pav}
A matroid $M$ of rank $n$ is called a {\em paving matroid} if every circuit of $M$ has a size of either $n$ or $n+1$. In this context, we refer to it as an $n$-paving matroid. Equivalently, any $n$-paving matroid can be derived from Definition~\ref{elementary} with the condition that every $r_{i}$ is equal to $n-1$.  For paving matroids, the set of subspaces $\mathcal{L}_{M}$, as described in Definition~\ref{general}, coincides with the set of {\em dependent hyperplanes} of $M$,~i.e,~maximal subsets of points of size at least $n$, where every subset of $n$ elements forms a circuit. These subspaces are referred to as dependent hyperplanes, and in the special case where $n=3$, we simply call them {\em lines}. 
\end{definition}

\begin{example}
The matroid illustrated in Figure~\ref{fig:combined} (Left) represents a paving matroid of rank $3$ with points $\{p_{1},\ldots,p_{7}\}$ and lines given by $\mathcal{L}=\{\{p_{1},p_{2},p_{7}\}, \{p_{3},p_{4},p_{7}\}, \{p_{5},p_{6},p_{7}\}\}$. 
\end{example}

We now describe the set of subspaces of an elementary split matroid.

\begin{lemma}\label{subspaces split}
Let $M$ be the elementary split matroid of 
$\mathcal{H}=\{H_{1},\ldots,H_{q}\}$. Then, we have $\mathcal{L}_M=\mathcal{H}$. 
\end{lemma}
\begin{proof}
It is enough to show that the circuits of $M$ of size at most $n$ consist of the $(r_{i}+1)$-subsets of $H_{i}$ for each $i\in [q]$. 
From \textup{\cite[Theorem~10]{berczi2023hypergraph}}, we know that each $(r_{i}+1)$-subset of $H_{i}$ is a circuit for every $i\in [q]$. Now, let $C=\{x_{1},\ldots,x_{k}\}$ be a circuit with $k\leq n$. Since $C$ is a dependent set, by Definition~\ref{elementary}, there must exist some $i\in [q]$ such that $\size{C\cap H_{i}}\geq r_{i}+1$. Therefore, $C$ must indeed be an $(r_{i}+1)$-subset of $H_{i}$, since $C$ is a circuit. Furthermore, by the same theorem, each $H_{i}$ is a flat of rank $r_{i}$, so the closure of each $(r_{i}+1)$-subset of $H_{i}$ coincides with $H_{i}$. This implies that in the equivalence relation of Equation~\eqref{equiv}, there is precisely one equivalence class associated with each ${H}_{i}$, composed of the $(r_{i}+1)$-subsets of $H_{i}$. Therefore, we conclude that $\mathcal{L}_M=\mathcal{H}$.
\end{proof}

\subsection{Realization spaces of matroids}
In this subsection, we recall the definitions of the realization space of a matroid, along with its associated variety and ideal. Throughout the entire paper, we assume the field is the complex numbers. 
\begin{definition}\normalfont
Let $M$ be a matroid of rank $n$ on the ground set $[d]$. 
A realization of $M$ 
is a collection of vectors $\gamma=\{\gamma_{1},\ldots,\gamma_{d}\}\subset \CC^{n}$ such that
\[\{\gamma_{i_{1}},\ldots,\gamma_{i_{p}}\}\ \text{is linearly dependent} \Longleftrightarrow \{i_{1},\ldots,i_{p}\} \ \text{is a dependent set of $M$.}\]
The {\em realization space} of $M$ 
is defined as 
$\Gamma_{M}=\{\gamma\subset \CC^{n}: \gamma \ \text{is a realization of $M$}\}$.
Each element of $\Gamma_{M}$ can be represented by an $n\times d$ matrix over $\CC$, or equivalently an element of $\CC^{nd}$. 
\end{definition}

\begin{definition}
    The {\em matroid variety} $V_M$ 
    is defined as the Zariski closure  of 
    $\Gamma_M$ in $\CC^{nd}$. 
\end{definition}

We now introduce the {\em circuit variety} of a matroid.

\begin{definition}\normalfont\label{cir}
Let $M$ be a matroid of rank $n$ on the ground set $[d]$. 
Consider the $n\times d$ matrix $X=\pare{x_{i,j}}$ of indeterminates. We define the {\em circuit ideal} of $M$ as 
$$ I_{\mathcal{C}(M)} = \{ [A|B]_X:\ B \in \mathcal{C}(M),\ A \subset [n],\ \text{and}\ |A| = |B| \}.$$
We say that a collection $\gamma$ of vectors of $\CC^{n}$ indexed by $[d]$, includes the dependencies of $M$ if it satisfies:
\[\set{i_{1},\ldots,i_{k}}\  \text{is a dependent set of $M$} \Longrightarrow \set{\gamma_{i_{1}},\ldots,\gamma_{i_{k}}}\ \text{is linearly dependent}. \] 
The {\em circuit variety} of $M$ is defined as $V_{\mathcal{C}(M)}=V(I_{\mathcal{C}(M)})=\{\gamma:\text{$\gamma$ includes the dependencies of $M$}\}$.
Moreover, we define the \textit{basis ideal} of $M$ as follows:
\begin{equation}\label{jm}
J_{M}=\sqrt{
\prod_{B\in \mathcal{B}\pare{M}}\ip{[A|B]_{X}:\ A\subset [n]}{\size{A}=\size{B}}.}
\end{equation}
A collection of vectors belongs to $V(J_{M})$ if and only if it has a dependent subset that is not in $\mathcal{D}(\mathcal{M})$. Consequently, we have $\Gamma_{M}=V_{\mathcal{C}(M)}\backslash V(J_{M})$.
\end{definition}

\subsection{Realization spaces of hyperplane arrangements}
In this subsection, we will first recall the correspondence between hyperplane arrangements and matroids, and then provide an overview of the relationship between their realization spaces.

\begin{definition}\normalfont\label{combinatorics}
Any finite collection $\mathcal{H}=\{H_{1},\ldots,H_{d}\}$ of hyperplanes in $\CC^{n}$ is referred to as a {\em hyperplane arrangement}. For the purposes of this work, we 
assume that the intersection of these hyperplanes is empty.
The {\em complement manifold} $N(\mathcal{H})$ is defined as the complement of the union of the hyperplanes $H_{i}$ in $\CC^{n}$. For each arrangement $\mathcal{H}=\{H_{1},\ldots,H_{d}\}$ in $\CC^{n}$, we 
associate a matroid $M_{\mathcal{H}}$ on the ground set $[d]$, with the rank function defined on any subset $S\subset \{1,\ldots,d\}$ as follows 
\[\textstyle{\rank_{M_{\mathcal{H}}}(S)=\text{codim}\bigcap_{i\in S}H_{i}.}\]
Two hyperplane arrangements $\mathcal{H}$ and $\mathcal{H}'$ are said to have the same \textit{combinatorics} if $M_{\mathcal{H}}=M_{\mathcal{H}'}$. 
The \textit{realization space} $\Gamma_{\mathcal{H}}$ of a hyperplane arrangement $\mathcal{H}$ is defined as the set of all the hyperplane arrangements $\mathcal{H}'$ in $\CC^{n}$ such that $M_{\mathcal{H}'}=M_{\mathcal{H}}$.
\end{definition}

\subsubsection{\small{Correspondence between the realization spaces of matroids and hyperplane arrangements}}\label{correspondence} 

Let $M$ be a matroid of rank $n$ on the ground set $[d]$, and let
$\gamma\in \Gamma_{M}$. This defines a hyperplane arrangement composed of the orthogonal hyperplanes to each vector in $\gamma$, whose 
associated matroid equals $M$. Conversely, consider a hyperplane arrangement $\mathcal{H}=\{H_{1},\ldots,H_{d}\}$ in $\CC^{n}$. The collection of vectors $\gamma$, consisting of the normal vectors to the hyperplanes $H_{i}$, lies within the realization space of the matroid $M_{\mathcal{H}}$. 
Therefore, for a realizable matroid $M$, there exists a correspondence between $\Gamma_{M}$ and $\Gamma_{\mathcal{H}}$, where $\mathcal{H}$ is an arbitrary hyperplane arrangement such that $M_{\mathcal{H}}=M$.

\medskip
Under the aforementioned correspondence, we present the following examples.
\begin{example}\label{corresp}
\begin{itemize}
\item The realization spaces of line arrangements correspond to those of point-line configurations, which are paving matroids of rank $3$ (see e.g. \cite{Emiliano-Fatemeh4}). 
\item The realization spaces of inductively connected line arrangements, as defined in \cite{nazir2012connectivity}, correspond to those of solvable matroids of rank $3$, as introduced in \cite{Emiliano-Fatemeh5}.
\end{itemize}
\end{example}

We now present the definition of matroid stratum $\text{Gr}(M,\CC)$. First, we recall that 
the Grassmannian $\text{Gr}(n, d)$ is the space of all $n$-dim linear subspaces of $\mathbb{\mathbb{C}}^d$. To emphasize the underlying field, we may write $\text{Gr}(n,d,\CC)$.  A point $V$ in $\text{Gr}(n,d)$ can be represented by an $n\times d$ matrix with entries in $\CC$. 
Let $X=(x_{ij})$ be an $n\times d$ matrix of indeterminates. For a subset $\lambda = \{\lambda_1,\ldots,\lambda_n\} \in \textstyle{\binom{[d]}{n}}$, let $X_\lambda$ denote the $n\times n$ submatrix of $X$ with the column indices $\lambda_1,\ldots,\lambda_n$. 
The \emph{Pl\"ucker coordinates} of $V$ are $p_\lambda(V) = \text{det}(X_\lambda)$ for $\lambda \in\textstyle{\binom{[d]}{n}}$. They do not depend on the choice of matrix $X$ (up to simultaneous rescaling by a non-zero constant) and determine the \emph{Pl\"ucker embedding} of
$\text{Gr}(n,d)$ into $\mathbb{P}^{\binom{d}{n}-1}$.

\begin{definition}\normalfont
Let $M$ be a matroid of rank $[n]$ on the ground set $[d]$.
The {\em matroid stratum}~of~$M$~is:
\[\text{Gr}(M,\CC):=\{V\in \text{Gr}(n,d,\CC): p_{\lambda}(V)\neq 0 \ \text{if and only if}\ \lambda \in \mathcal{B}(M) \},\]
Fixing a reference basis $\lambda\in \mathcal{B}(M)$, we can view $\text{Gr}(M,\CC)$ as a quasi-affine variety, consisting of all $n\times d$ matrices (or equivalently, collections of vectors) of the form $(\Id_{n}|A)$ that realize $M$. 
\end{definition}

We now 
establish 
the precise relationship between $\text{Gr}(M,\CC)$ and $\Gamma_{M}$.

\begin{lemma}\label{igual 2}\normalfont
Let $M$ be a matroid of rank $n$ on the ground set $[d]$. Then, 
$\Gamma_{M} \cong \text{Gl}_{n}(\CC)\times \text{Gr}(M,\CC).$
Moreover, there is a 
bijection between the irreducible and connected components of~$\text{Gr}(M,\CC)$~and~$\Gamma_{M}$.
\end{lemma}

\begin{proof}
Consider the polynomial map 
\[F:\text{Gl}_{n}(\CC)\times \text{Gr}(M,\CC)\rightarrow \Gamma_{M}\quad\text{with}\quad
(A,(\Id_{n}|B))\mapsto (A|AB).
\]
Given that $\Gamma_{M}$ is stable under the action of $\text{Gl}_{n}(\CC)$, it is clear that $F$ is a bijection. Moreover, the inverse map $F^{-1}$, given by
$F^{-1}(A|B)=(A,(\Id_{n}|A^{-1}B))$
is also a polynomial map, thereby establishing the desired isomorphism.
Moreover, since $\text{Gl}_{n}(\CC)$ is irreducible, and in particular connected, there exists a correspondence between the irreducible and connected components of $\text{Gr}(M,\CC)$ and $\Gamma_{M}$, arising from the projection onto the second coordinate of $\text{Gl}_{n}(\CC)\times \text{Gr}(M,\CC)$.
\end{proof}

\begin{remark}\label{relation between varieties}
Let $M$ be a matroid of rank $n$ on the ground set $[d]$. Observe that our definition of the realization space $\Gamma_{M}$ aligns with that in \cite{clarke2021matroid,Emiliano-Fatemeh4}, where the ambient space is $\CC^{nd}$. However, it differs from the one in \cite{gelfand1987combinatorial, sidman2021geometric,ford2015expected},  where the realization space of $M$ is defined as $\text{Gr}(M,\CC)$, with ambient space $\text{Gr}(n,d)$. However, by Lemma~\ref{igual 2}, we have:
\[\dim(\Gamma_{M})=\dim(\text{Gr}(M,\CC))+n^{2}, \quad \text{and} \quad \dim(\text{Gr}(n,d))=n(d-n),\]
which implies that:
\[\text{codim}(\Gamma_{M})=\text{codim}(\text{Gr}(M,\CC)),\]
where the codimensions are taken relative to the ambient spaces $\CC^{nd}$ and $\text{Gr}(n,d)$, respectively.
\end{remark}

We will now discuss the motivation behind studying the connectivity of these spaces.

\begin{definition}
\normalfont A smooth one-parameter family of arrangements $\{\mathcal{H}_{t}:t\in I\}$ on an interval $I$ is called an \textit{isotopy} if, for any $t_{1},t_{2}\in I$, the arrangements $\mathcal{H}_{t_{1}}$ and $\mathcal{H}_{t_{2}}$ have 
the same combinatorial type. In this case, we say that $\mathcal{H}_{t_{1}}$ and $\mathcal{H}_{t_{2}}$ are \textit{isotopic}.
\end{definition}

The following theorem, commonly referred to as “Randell’s Lattice Isotopy Theorem”, clarifies the connection between the connectivity of realization spaces and the topology of their complements. 

\begin{theorem}[\textup{\cite{randell1989lattice}}]\label{ran}
 If the arrangements $\mathcal{H}_{t_1}$ and $\mathcal{H}_{t_2}$ are isotopic, then their complement manifolds, $N(\mathcal{H}_{t_1})$
 and $N(\mathcal{H}_{t_2})$, are diffeomorphic.
\end{theorem}

As a consequence of Theorem~\ref{ran}, and given that any connected space defined by a system of polynomial equalities and inequalities is path connected we derive the following corollary. 

 \begin{corollary}\label{coro ar}
Let $\mathcal{H}$ be a hyperplane arrangement. If the arrangements $\mathcal{H}_{1}$ and $\mathcal{H}_{2}$ belong to the same irreducible component of $\Gamma_{\mathcal{H}}$, then  $N(\mathcal{H}_{1})\cong N(\mathcal{H}_{2})$.
\end{corollary}

\subsection{Grassmann-Cayley algebra}\normalfont
We now recall some results from \textup{\cite{computationalgorithms}}. The Grassmann-Cayley algebra is defined as the exterior algebra $\textstyle \bigwedge (\CC^{d})$ equipped with two operations: the \textit{join} and the \textit{meet}. We denote the join of vectors $v_{1},\ldots ,v_{k}$ as $v_{1}\vee \ldots \vee v_{k}$, or simply by $v_{1}\cdots v_{k}$, referring to it as an extensor. The meet operation, denoted by $\wedge$, is defined on two extensors $v=v_{1}\cdots v_{k}$ and $w=w_{1}\cdots w_{j}$, where $j+k\geq d$, as follows:

\begin{equation}\label{equ wedge}
v\wedge w=\sum_{\sigma\in \mathcal{S}(k,j,d)}
\corch{v_{\sigma(1)}\cdots v_{\sigma(d-j)}w_{1}\cdots w_{j}}\cdot v_{\sigma(d-j+1)}\cdots v_{\sigma(k)}.
\end{equation}
Here, $ \mathcal{S}(k,j,d)$ denotes the set of all permutations of $\corch{k}$ such that $\sigma(1)<\cdots <\sigma(d-j)$ and $\sigma(d-j+1)<\cdots <\sigma(d)$. When $j+k<d$, the meet is defined to be $0$. There is a correspondence between the extensor $v=v_{1}\cdots v_{k}$ and the subspace 
generated by $\{v_{1},\ldots,v_{k}\}$, which satisfies the following properties:

\begin{lemma}\label{klj}

Let $v=v_{1}\cdots v_{k}$ and $w=w_{1}\cdots w_{j}$ be two extensors with $j+k\geq d$. Let $\overline{v}$ denote the subspace $\ip{v_{1},\ldots}{v_{k}}$.
Then, the following statements hold:

\begin{itemize}
\item The extensor $v$ is zero if and only if the vectors $v_{1},\ldots,v_{k}$ are linearly dependent.
\item Any extensor $v$ is uniquely determined by the subspace $\overline{v}$ it generates, up to a scalar multiple.
\item The meet of two extensors is itself an extensor.
\item The meet $v\wedge w$ is nonzero if and only if $\ip{\overline{v}}{\overline{w}}=\CC^{d}$. In this case, we have  $\overline{v}\cap\overline{w}=\overline{v\wedge w}.$
\end{itemize}
\end{lemma}

\begin{example}\label{gc3}
Consider the 
matroid $M$ in Figure~\ref{fig:combined} (Left). Let $\{p_{1},\ldots,p_{7}\}\subset \CC^{3}$ be a collection of vectors realizing this matroid.
Since the lines $\{p_{1}p_{2},p_{3}p_{4},p_{5}p_{6}\}$ are concurrent, the intersection point of the lines $p_{1}p_{2}$ and $p_{3}p_{4}$ lies on the line $p_{5}p_6$. Therefore, the condition for  
$\{p_1 p_2,p_3 p_4,p_5 p_6\}$ to be concurrent is given by the vanishing of the following expression:
$$p_3 p_4 \wedge p_1 p_2 \vee p_5 p_6 = (\corch{p_3 p_1 p_2} p_4 - \corch{p_4 p_1 p_2} p_3) \vee p_5 p_6 = \corch{p_1 p_2 p_3} \corch{p_4 p_5 p_6} - \corch{p_1 p_2 p_4} \corch{p_3 p_5 p_6}.$$
Thus, this polynomial is one of the defining equations of $\Gamma_M$. More specifically, this polynomial, together with the three polynomials associated with the circuits, defines the matroid.
\end{example}

We conclude this subsection with the following lemma, \cite[Lemma~5.29]{Emiliano-Fatemeh5}, which is pivotal in \S\ref{sec:realization}.

\begin{lemma}\label{basis}
Let $S,T\subset \CC^{n}$ be subspaces of dimensions $k$ and $n-k$, respectively, such that $S\cap T=\set{0}$. Suppose that $v_{1},\ldots,v_{k}$ are vectors such that $T+\is{v_{1},\ldots,v_{k}}=\CC^{n}$. For each $i\in [k]$, let $t_{i}$ be an extensor associated with the subspace $T+\is{v_{i}}$ and let $s$ be an extensor associated with $S$. Then, the vectors 
$\set{t_{i}\wedge s: i\in [k]}$
form a basis of $S$.
\end{lemma}

\section{Naive dimension of the realization space}\label{sec 3}

We define the notion of naive dimension for the realization space of matroids, which was introduced in \cite{guerville2023connectivity} for line arrangements, or matroids of rank $3$, and we will extend it to matroids of arbitrary rank.

\begin{definition}[Naive dimension]\normalfont \label{naiv defi}
Let $M$ be a matroid of rank $n$ on the ground set $[d]$, with the set of subspaces $\mathcal{L}_{M}$. 
We define the \textit{naive dimension} of the realization space $\Gamma_{M}$ as follows: 
\begin{equation}\label{ecuacion naive} \dnaive( \Gamma_{M})=nd-\sum_{l\in \mathcal{L}_M}(\size{l}-\rank(l))(n-\rank(l)).
\end{equation}
\end{definition}

\begin{remark}
The right-hand side of~\eqref{ecuacion naive} 
can be viewed as a naive attempt to define the dimension of the realization space  $\Gamma_{M}$. The number $nd$ corresponds to the number of variables in the $n\times d$ matrix of indeterminates $X=(x_{i,j})$. Furthermore, the condition that each collection of $\rank(l)+1$ vectors from $l\in\mathcal{L}_M$ is dependent can be viewed as the vanishing of $(\size{l}-\rank(l))(n-\rank(l))$-minors of $X$.
\end{remark}

Now we will prove that the \textit{naive dimension} serves as a lower bound for the dimension of $\Gamma_{M}$.

\begin{theorem}\label{bound on dimension}
Let $M$ be a realizable matroid of rank $n$ on $[d]$. Then, we have
$\dnaive(\Gamma_{M})\leq \dim(\Gamma_{M})$.
\end{theorem}

\begin{proof}
Let $X=(x_{ij})$ be an $n\times d$ matrix of indeterminates. For each subspace $l\in \mathcal{L}_M$, we fix a basis $B_{l}=\{q_{l,1},\ldots,q_{l,\rank(l)}\}$. 
For any collection $C$ of subsets $(C_{l})_{l\in \mathcal{L}_M}\in \textstyle \prod_{l\in \mathcal{L}_M}\textstyle \binom{[n]}{\rank(l)}$, we define $\Gamma_{M,C}$ to be the subset of realizations $\gamma \in \Gamma_{M}$ such that $\det([C_{l}|B_{l}]_{X})$ does not vanish on the vectors of $\gamma$. Since the vectors $\{\gamma_{p}:p\in B_{l}\}$ are independent for each $l\in \mathcal{L}_M$, it follows that for each $\gamma \in \Gamma_{M}$ and each $l\in \mathcal{L}_M$, there exists $C_{l}\in \textstyle \binom{[n]}{\rank(l)}$ such that $\det([C_{l}|B_{l}]_{X})$ does not vanish on the vectors of $\gamma$. Therefore, we obtain the open cover
\begin{equation}\label{union}
\Gamma_{M}=\bigcup_{C} \Gamma_{M,C}, 
\end{equation}
where the union is taken over $\textstyle \prod_{l\in \mathcal{L}_M}\textstyle \binom{[n]}{\rank(l)}$. Now we fix a specific $C$, and aim to bound the dimension of $\Gamma_{M,C}$. Note that for any $C_{l}\rq \times B_{l}\rq \in \textstyle \binom{[l]}{\rank(l)+1}\times \binom{[n]}{\rank(l)+1}$, where $[l]$ denotes the elements of $l$, we have that $\det([C_{l}\rq|B_{l}\rq]_{X})$ vanishes on $\Gamma_{M}$, since any collection of $\rank(l)+1$ columns associated with elements of $l$ are dependent in $\Gamma_{M}$. On the other hand, since $\det([C_{l}|B_{l}]_{X})\neq 0$ on $\Gamma_{M,C}$, the vanishing of the minors $\det([C_{l}\rq|B_{l}\rq]_{X})$ for $C_{l}\rq \times B_{l}\rq \in \textstyle \binom{[l]}{\rank(l)+1}\times \binom{[n]}{\rank(l)+1}$ with $C_{l}\subset C_{l}\rq$ and $B_{l}\subset B_{l}\rq$ serves as a sufficient condition for the columns associated to $l$ to have rank at most $\rank(l)$. We denote this family of minors by $K_{l}$, which consists of $(\size{l}-\rank(l))(n-\rank(l))$ polynomials. Since $\Gamma_{M}$ is an open set of $V_{\mathcal{C}(M)}$, it follows that $\Gamma_{M,C}$ is an open Zariski subset of 
$\textstyle{\bigcap_{l\in \mathcal{L}_M}V(K_{l})}$.

Consequently, $\Gamma_{M,C}$ is characterized as the intersection of at most $\textstyle \sum_{l\in \mathcal{L}_M}(\size{l}-\rank(l))(n-\rank(l))$ hypersurfaces within an open Zariski subset of the affine space of dimension $nd$. Then, by applying well-known results about the dimension of hypersurface intersections, we conclude that \[\dnaive(\Gamma_{M})\leq \dim(\Gamma_{M,C})  \quad  \text{or} \quad \Gamma_{M,C}=\emptyset.\]  Then, using the open cover of $\Gamma_{M}$ as in Equation~\eqref{union}, we obtain that $\dnaive(\Gamma_{M})\leq \dim(\Gamma_{M})$.
\end{proof}

We now present two families of examples, namely the elementary split matroids and the paving matroids, where we can explicitly compute the naive dimension. We recall our notation from \S\ref{sec:split}.

\begin{lemma}\label{lema split}
Let $M$ be an elementary split matroid of 
$\mathcal{H}=\{H_{1},\ldots,H_{q}\}$ as in Definition~\ref{elementary}. Then,  
\[\dnaive(\Gamma_{M})=nd-\sum_{i=1}^{q}(n-r_{i})(\size{H_{i}}-r_{i}).\]
\end{lemma}

\begin{proof}
The equality follows directly from Lemma~\ref{subspaces split}.
\end{proof}

If $M$ is an $n$-paving matroid, the formula of the {\em naive dimension} of its realization space can be expressed in terms of the number of its dependent hyperplanes and the degrees of its points.

\begin{lemma}\label{na pav}\normalfont
Let $M$ be an $n$-paving matroid on $[d]$. Then, we have 
\begin{equation}\label{naive pav} \dnaive(\Gamma_{M})
=nd+(n-1)\size{\mathcal{L}_M}-\sum_{p\in \mathcal{P}_M}\text{deg}(p).\end{equation}
\end{lemma}

\begin{proof}
Observe that $\textstyle{\sum_{l\in \mathcal{L}_M}\size{l}=\sum_{p\in \mathcal{P}_M}\text{deg}(p)}$.
Substituting this equation into Equation~\eqref{ecuacion naive} and using the fact that every dependent hyperplane has rank $n-1$, we arrive at the desired equality.
 \end{proof}

\subsection{Relation between naive dimension and expected codimension}
In \cite{ford2015expected}, Ford introduced a recursive formula for calculating the expected dimension of matroid varieties, which are the Zariski closures of realization spaces. He also demonstrated that, for the class of positroid varieties, this formula equals the actual dimension. In this subsection, we establish that the naive dimension and expected dimension are equal for the family of elementary split matroids. 

\medskip

We will first recall the definition of expected dimensions; see \cite[Definition~1.1]{ford2015expected}.

\begin{definition}
\normalfont
\label{exp dimension}
Let $M$ be a matroid of rank $n$ on 
$[d]$, $\mathcal{P}([d])$ the collection of all subsets of $[d]$, and $\mathcal{G}$ a subset of  $\mathcal{P}([d])$. 
For each $S\in \mathcal{G}$, let $c(S)=\size{S}-\rank(S)$. Then, 
\begin{itemize}
    \item We recursively define
$a_{\mathcal{G}}(S)=c(S)-\sum_{\substack{T\in \mathcal{G}\\ T\subsetneq S}}a_{\mathcal{G}}(T),$ where $a_{\mathcal{G}}(\emptyset)=0$.

\item The {\em expected codimension} of $M$ with respect to $\mathcal{G}$ is defined as
\begin{equation}\label{ecg}\text{ec}_{\mathcal{G}}(M)=\sum_{S\in \mathcal{G}}(n-\rank(S))a_{\mathcal{G}(S)},\end{equation}
and the {\em expected codimension} of $M$ is given by 
$\text{ec}(M)=\text{ec}_{\mathcal{P}([d])}(M)$.
\item The {\em expected dimension} of $M$ is defined as 
$\text{ed}(M)=nd-\text{ec}(M).$
\end{itemize}
\end{definition}

We now analyze the behavior of expected codimension under direct sums; see \cite[Proposition~3.7]{ford2015expected}.\footnote{The original proposition in \cite{ford2015expected} omits the constant summands from \eqref{correct}, stating that $\text{ec}(M\oplus N)=\text{ec}(M)+\text{ec}(N)$. This error occurs between the first and second lines in the chain of equalities in the proof of Proposition~3.7 in \cite{ford2015expected}, where the codimension is incorrectly evaluated for both $M \oplus N$ and the individual matroids $M$ and $N$, rather than consistently considering only $M \oplus N$. However, this oversight can be easily corrected to obtain the accurate formula in Equation~\ref{correct}.} 

\begin{lemma}\label{direct sum}
Let $M$ and $N$ be two matroids of ranks $n_{1}$ and $n_{2}$, on $[d_{1}]$ and $[d_{2}]$, respectively. Then, 
\begin{equation}\label{correct}
\text{ec}(M\oplus N)=\text{ec}(M)+\text{ec}(N)+n_{2}(d_{1}-n_{1})+n_{1}(d_{2}-n_{2}).\end{equation}
\end{lemma}

\begin{corollary}
    If the codimensions of $\Gamma_{M}$ and $\Gamma_{N}$ equal their expected codimensions, then this also holds for $M \oplus N$.
\end{corollary}
\begin{proof}
    Since $\text{Gr}(M\oplus N,\CC)\cong \text{Gr}(M,\CC)\oplus \text{Gr}(N,\CC)$, it is straightforward to verify that
\begin{equation}\label{seco}\text{codim}(\Gamma_{M\oplus N})=\text{codim}(\Gamma_{M})+\text{codim}(\Gamma_{N})+n_{2}(d_{1}-n_{1})+n_{1}(d_{2}-n_{2}).\end{equation}
Hence, the result follows by applying Equations~\eqref{correct} and~\eqref{seco}. 
\end{proof}

\begin{remark}
The formula for $\text{ec}(M)$ in \cite{ford2015expected} serves as an attempt to predict the codimension of the matroid variety $X(M)$, which is defined there as the Zariski closure of $\text{Gr}(M,\CC)$ within $\text{Gr}(n,d)$. Thus, $\text{ec}(M)$ aims to estimate $\text{codim}(X(M))$ in $\text{Gr}(n,d)$, which equals $\text{codim}(\text{Gr}(M,\CC))$. Nevertheless, using Remark~\ref{relation between varieties}, $\text{ec}(M)$ can also be interpreted as the expected codimension of $\Gamma_{M}$ within $\CC^{nd}$.
\end{remark}

The following lemma will be essential in what follows. 

\begin{lemma}\label{conectados}
Let $M$ be the connected elementary split matroid 
of a hypergraph $\mathcal{H}=\{H_{1},\ldots,H_{q}\}$. 
Suppose for some $S\subset [d]$, 
 both $S$ and $M/S$ are connected, and 
 $\rank(S)<n$.
Then, $\size{S}\leq 1$ or $S=H_{i}$ for some $i\in [q]$.
\end{lemma}

\begin{proof}
We consider the following cases:

\medskip
{\bf Case~1.} Suppose that $\size{S\cap H_{i}}\leq r_{i}$ for all $i$. By definition, this implies that $\rank(S)=\text{min}\{n,\size{S}\}$. Given that $\rank(S)<n$, we conclude that $\rank(S)=\size{S}$. Since $S$ is connected, it follows that $\size{S}\leq 1$.

\medskip
{\bf Case~2.} Suppose that $\size{S\cap H_{i}}> r_{i}$ for some $i\in [q]$. In this case, since the closure of any $(r_{i}+1)$-subset of $H_{i}$ coincides with $H_{i}$, it follows that all the points in $H_{i}\backslash S$ become loops in $M/S$. Since $M/S$ is connected, this situation leads to a contradiction unless $\size{S}=d-1$ or $H_{i}\subset S$. We will now analyze these possibilities.

\medskip
{\bf Case~2.1.} If $\size{S}=d-1$, then we have $\rank(S)=n$ since $M$ is connected.

\medskip
{\bf Case~2.2.} If $H_{i}\subset S$, by applying \textup{\cite[Theorem~10]{berczi2023hypergraph}}, we obtain that $M/H_{i}\cong U_{n-r_{i},d-\size{H_{i}}}$. This implies that $\rank(S)=r_{i}+\min\{n-r_{i},\size{S\backslash H_{i}}\}$.
Given that $\rank(S)<n$, it follows that $\size{S\backslash H_{i}}<n-r_{i}$. Consequently, we have $\rank(S)=r_{i}+\size{S\backslash H_{i}}=\rank(H_{i})+\size{S\backslash H_{i}}$, indicating that the points in $S\backslash H_{i}$ are coloops in $S$. Since $S$ is connected, we conclude that $S=H_{i}$.
\end{proof}

We now show that, for the family of elementary split matroids, the expected dimension and the naive dimension are equal.

\begin{theorem}\label{teo coincid}
Let $M$ be an elementary split matroid of a 
hypergraph $\mathcal{H}=\{H_{1},\ldots,H_{q}\}$. 
Then,
\[\dnaive(\Gamma_{M})=nd-\textup{ec}(M)=nd-\sum_{i=1}^{q}(n-r_{i})(\size{H_{i}}-r_{i}).\]
\end{theorem}

\begin{proof}
We will prove the statement by considering the following cases:

\medskip
{\bf Case~1.} Suppose $M$ is connected. By applying \textup{\cite[Theorem~3.6]{ford2015expected}}, we know that $\text{ec}(M)=\text{ec}_{\mathcal{G}}(M)$, where $\mathcal{G}$ is the collection of all subsets of $[d]$ such that both $S$ and $M/S$ are connected. Using Lemma~\ref{conectados}, we can express $\mathcal{G}$ as 
\[\mathcal{G}=\mathcal{H}\cup \mathcal{A} \cup \mathcal{B},\quad\text{where}\quad\mathcal{A}=\{S\subset [d]:\size{S}\leq 1\}, \quad \text{and} \quad \mathcal{B}=\{S\in \mathcal{G}:\rank(S)=n\}.\]
From Equation~\eqref{ecg}, it follows that the subsets of $\mathcal{G}$ of rank $n$ do not contribute to $\text{ec}_{\mathcal{G}}(M)$. Therefore, 
 $\text{ec}_{\mathcal{G}}(M)=\text{ec}_{\mathcal{G}\backslash \mathcal{B}}(M).$ 
 Moreover, since $a_{\mathcal{G}\backslash \mathcal{B}}(S)=0$ for every $S\in \mathcal{A}$, by \textup{\cite[Corollary~3.3]{ford2015expected}} we have that $\text{ec}_{\mathcal{G}\backslash \mathcal{B}}(M)=\text{ec}_{\mathcal{H}}(M)$. Thus, we conclude
$\textstyle{\text{ec}(M)=\text{ec}_{\mathcal{H}}(M)=\sum_{i=1}^{q}(n-r_{i})(\size{H_{i}}-r_{i})}.$
Finally, the result follows from Lemma~\ref{lema split}.

\medskip
{\bf Case~2.} Suppose $M$ is disconnected. According to \textup{\cite[Theorem~11]{berczi2023hypergraph}}, we know that $M$ is the direct sum of two uniform matroids, i.e.,  $M=U_{n_{1},d_{1}}\oplus U_{n_{2},d_{2}}$. From Equation~\eqref{ecuacion naive}, it follows that
\[\dnaive(\Gamma_{M})=(d_{1}+d_{2})(n_{1}+n_{2})-n_{2}(d_{1}-n_{1})-n_{1}(d_{2}-n_{2}).\]
Moreover, it is easy to verify that $\text{ec}(U_{n_{1},d_{1}})=\text{ec}(U_{n_{2},d_{2}})=0$. Then, by Equation~\ref{correct}, we have that \begin{equation}\label{equi}\text{ec}(U_{n_{1},d_{1}}\oplus U_{n_{2},d_{2}})=n_{2}(d_{1}-n_{1})+n_{1}(d_{2}-n_{2}),\end{equation} which completes the proof. 
\end{proof}

\section{Inductively connected matroids}\label{sec 4}
In this section, we introduce the family of inductively connected matroids of arbitrary rank. The main result, presented in Theorem~\ref{gen}, establishes the irreducibility and smoothness of the realization spaces for these matroids, specifically demonstrating their connectivity. Additionally, we provide an explicit characterization of the realization space of an inductively connected matroid as an open Zariski subset of a complex affine space. This characterization generalizes the result in \textup{\cite[Proposition~3.3]{nazir2012connectivity}}, which applies to line arrangements, or equivalently matroids of rank 3.

\subsection{Type of an ordering} 
In this subsection, we introduce the notion of the type of an ordering within a matroid's ground set. This concept allows us to establish an upper bound on the naive dimension of a matroid and to calculate the naive dimension specifically for elementary split matroids.

\medskip

\begin{definition}\normalfont\label{def ordering}
Let $M$ be a matroid with $\size{\mathcal{P}_M}=d$. An \textit{ordering} of $M$ is a bijective map
\[w:\mathcal{P}
_M\rightarrow \{1,\ldots,d\}.\]
We will often denote this ordering as $(p_{1},\ldots,p_{d})$, where $w(p_{j})=j$ for all $j\in [d]$, meaning the indices align with the ordering $w$.
\end{definition}

Any ordering $w$ naturally induces an ordering $w_{N}$ on any subset $N\subset \mathcal{P}_M$, where for all $p,q\in N$,
\[w_{N}(p)<w_{N}(q)\Longleftrightarrow w(p)<w(q).\]

\begin{definition}\label{defi f}\normalfont
Let $M$ be a matroid of rank $n$ with $\size{\mathcal{P}_M}=d$. 
\begin{itemize}
\item For any ordering $w$ and any $i\in [d]$, let 
$\mathcal{P}_{w}[i]=w^{-1}(\{1,\ldots,i\}).$

\item We also denote by $M_{w}[i]$ the submatroid of $M$ on the ground set $\mathcal{P}_{w}[i]$ and let $\mathcal{L}_{w}[i]$ be the set of its subspaces. When $w$ is given by an ordering of the points, we have $\mathcal{P}_{w}[i]=\{p_{1},\ldots,p_{i}\}$. We denote ${(\mathcal{L}_{w}[i])}_{p_{i}}$ for the set of subspaces of $M_{w}[i]$ containing the point $p_{i}$.  

\item We define the $i^{\text{th}}$ type of an ordering as follows:
\begin{equation}\label{equ tau}
\tau_{i}(M,w)=\sum_{l\in {(\mathcal{L}_{w}[i])}_{p_{i}}}\rank(l) -n(\lvert{{(\mathcal{L}_{w}[i])}_{p_{i}}}\rvert-1).
\end{equation}
For simplicity, when the context is clear, we simply denote this number as $\tau_{i}$.

\item The \textit{type} of an ordering $(M,w)$ is the tuple 
$\tau(M,w)=(\tau_{1},\ldots,\tau_{d})\in \mathbb{Z}^{d}$.
\end{itemize}
\end{definition}

\begin{example} 
For an $n$-paving matroid $M$, we have $\tau_{i}=n-\text{deg}(p_{i})$, where $\text{deg}(p_{i})$ is the number of dependent hyperplanes of $M_{w}[i]$ containing the point $p_{i}$.
\medskip

For instance, let $M$ be the $4$-paving matroid on the ground set $[7]$ with dependent hyperplanes: \[\mathcal{L}_M=\{\{1,2,3,4\},\{1,2,5,6\},\{3,4,5,6\},\{1,3,5,7\},\{2,4,6,7\}\}.\]
For the ordering $w=(1,2,3,4,5,6,7)$, the type of $(M,w)$ is $\tau(M,w)=(4,4,4,3,4,2,2)$.
\end{example}

The next proposition establishes the connection between the type of an ordering and the naive dimension of the realization space.

\begin{proposition}\label{igual 3}
Let $M$ be a matroid of rank $n$ on $[d]$, and let $w=(p_{1},\ldots,p_{d})$ be an ordering of $M$ such that $\{p_{1},\ldots,p_{n}\}$ forms a basis. Then, the following inequality holds:
\begin{equation}\label{equ na}
\dnaive(\Gamma_{M})\leq \sum_{i=1}^{d}\tau_{i}. 
\end{equation}
Moreover, if $M$ is elementary split, then the equality holds.
\end{proposition}
\begin{proof}
We proceed by induction on $d$. If $d=n$, the equality holds trivially. Now, suppose that $M$ has $d$ points and the result is true for any matroid with at most $d-1$ points. Let $M\rq$ be the submatroid with points $(p_{1},\ldots,p_{d-1})$. 
Note that any subspace $l\in \mathcal{L}_{M\rq}$ uniquely corresponds to a subspace $\widetilde{l}\in \mathcal{L}_{M}$.

\medskip
There are two possible cases for a subspace $l\in \mathcal{L}_{M\rq}$:

\medskip
{\bf Case~1.} There is no circuit $C$ containing $p_{d}$ such that $\closure{C}=\closure{l}$, that is the closure of $C$ and $l$ are equal. In this case $\widetilde{l}=l$. We denote the set of subspaces of this type by $\mathcal{A}_{1}$.

\medskip
{\bf Case~2.} There exists a circuit $C$ containing $p_{d}$ such that $\closure{C}=\closure{l}$. In this case, the corresponding subspace $\widetilde{l}\in \mathcal{L}_{M}$ contains $p_{d}$, and $\lvert\widetilde{l}\rvert \geq \size{l}+1$. We denote by $\mathcal{A}_{2}$ the set of subspaces of this type.

\medskip
There are three possible cases for a subspace $l\rq \in \mathcal{L}_{M}$:

\medskip
{\bf Case~1.} $l\rq \not \in \mathcal{L}_{p_{d}}$. In this case, there exists $l\in \mathcal{A}_{1}$ with $\widetilde{l}=l\rq$. We denote the set of subspaces of this type by $\mathcal{B}_{1}$.

\medskip
{\bf Case~2.} $l\rq \in \mathcal{L}_{p_{d}}$ and there exists a circuit $C$ not containing $p_{d}$ such that $\closure{C}=\closure{l\rq }$. In this case, there exists $l\in \mathcal{A}_{2}$ with $\widetilde{l}=l\rq$. We denote the set of subspaces of this type by $\mathcal{B}_{2}$.

\medskip
{\bf Case~3.} $l\rq \in \mathcal{L}_{p_{d}}$ and there is no circuit $C$ not containing $p_{d}$ such that $\closure{C}=\closure{l\rq }$. We denote the set of subspaces of this type by $\mathcal{B}_{3}$.

\medskip
By the induction hypothesis and Equation~\eqref{equ tau}, we have:
\begin{equation}\label{rec}
\begin{aligned}
\sum_{i=1}^{d}\tau_{i}&\geq n(d-1)-\sum_{l\in \mathcal{L}_{M\rq}}(\size{l}-\rank(l))(n-\rank(l))+(n-\sum_{l\in \mathcal{L}_{p_{d}}}n-\rank(l))\\
&=nd-\sum_{l\in \mathcal{A}_{2}}(\size{l}-\rank(l))(n-\rank(l))-\sum_{l\in \mathcal{A}_{1}}(\size{l}-\rank(l))(n-\rank(l))-\sum_{l\in \mathcal{B}_{2}\cup \mathcal{B}_{3}}n-\rank(l)\\
&\geq nd-\sum_{l\in \mathcal{A}_{2}}(\lvert \widetilde{l}\rvert -\rank(l)-1)(n-\rank(l))-\sum_{l\in \mathcal{B}_{1}}(\size{l}-\rank(l))(n-\rank(l))-\sum_{l\in \mathcal{B}_{2}\cup \mathcal{B}_{3}}n-\rank(l)\\
&\geq  nd-\sum_{l\in \mathcal{B}_{2}}(\lvert l\rvert -\rank(l))(n-\rank(l))-\sum_{l\in \mathcal{B}_{1}}(\size{l}-\rank(l))(n-\rank(l))-\sum_{l\in \mathcal{B}_{3}}(\size{l}-\rank(l))(n-\rank(l))\\
&= nd-\sum_{l\in \mathcal{L}_{M}}(\size{l}-\rank(l))(n-\rank(l)).
\end{aligned}
\end{equation}
Note that the second inequality above holds because the size of each subspace in $\mathcal{A}_{2}$ increases by at least one when adding the point $p_{d}$, and the third inequality follows from $\size{l}\geq \rank(l)+1$ for each subspace $l\in \mathcal{B}_{3}$. This completes the inductive step.

Now, suppose that $M$ is an elementary split matroid of a hypergraph $\mathcal{H}=\{H_{1},\ldots,H_{q}\}$, as described in Definition~\ref{elementary}. We now introduce the following partition of the hypergraph:
\begin{equation*}
\begin{aligned}
&\mathcal{H}_{1}=\{H_{i}: p_{d}\not \in H_{i}\}, \ \mathcal{H}_{2}=\{H_{i}: p_{d}\in H_{i}\ \text{and $\size{H_{i}}\geq r_{i}+2$}\},\ 
\mathcal{H}_{3}=\{H_{i}:p_{d}\in H_{i}\ \text{and $\size{H_{i}}= r_{i}+1$}\}.
 \end{aligned}
 \end{equation*}
By \textup{\cite[Theorem~8]{berczi2023hypergraph}}, the matroid $M\rq$ is also elementary split, defined by the hypergraph \[\mathcal{H}\rq=\{H_{i}\backslash \{p_{d}\}: \size{H_{i}\backslash \{p_{d}\}}\geq r_{i}+1\}=\mathcal{H}_{1}\cup \{H_{i}\backslash \{p_{d}\}:H_{i}\in \mathcal{H}_{2}\}.\]
Additionally, by Lemma~\ref{subspaces split}, 
$\mathcal{L}_{M}=\mathcal{H}$ and $\mathcal{L}_{M\rq}=\mathcal{H}\rq$.
By induction hypothesis and \eqref{equ tau}, we have 
\begin{equation*}
\begin{aligned}
\sum_{i=1}^{d}\tau_{i}&= n(d-1)-\sum_{l\in \mathcal{H}\rq}(\size{l}-\rank(l))(n-\rank(l))+(n-\sum_{H_{i}\in \mathcal{H}_{2}\cup \mathcal{H}_{3}}n-r_{i})\\
&= nd-\sum_{H_{i}\in \mathcal{H}_{1}}(\size{H_{i}}-r_{i})(n-r_{i})-\sum_{H_{i}\in \mathcal{H}_{2}}(\size{H_{i}}-r_{i})(n-r_{i})-\sum_{H_{i}\in \mathcal{H}_{3}}(\size{H_{i}}-r_{i})(n-r_{i})\\
&= nd-\sum_{H_{i}\in \mathcal{H}}(\size{H_{i}}-r_{i})(n-r_{i}),
\end{aligned}
\end{equation*}
where the penultimate equality holds since $\size{H_{i}}=r_{i}+1$ for each $H_{i}\in \mathcal{H}_{3}$.
Consequently, the equality holds at each step of the induction.

For elementary split matroids, the sets $\mathcal{B}_{i}$ correspond precisely to the sets $\mathcal{H}_{i}$ for $i\in [3]$, while the sets $\mathcal{A}_{1}$ and $\mathcal{A}_{2}$ coincide with $\mathcal{H}_{1}$ and $\{H_{i}\backslash \{p_{d}\}:H_{i}\in \mathcal{H}_{2}\}$, respectively. The reason why the equality holds for elementary split matroids, but not in general, lies in the fact that the size of each subspace $H_{i}\backslash \{p_{d}\}$ for $H_{i}\in \mathcal{H}_{2}$ increases by exactly one when adding $p_{d}$, and that for every $H_{i}\in \mathcal{H}_{3}$, the size of $H_{i}$ is exactly $r_{i}+1$, which causes the inequalities in Equation~\eqref{rec} to become equalities.
\end{proof}

\subsection{Inductively connected matroids}

In this subsection, we introduce the family of inductively connected matroids and establish some of their fundamental properties.

\begin{definition}\normalfont\label{def simp}
Let $M$ be a matroid of rank $n$ on $[d]$. We say that $M$ is \textit{inductively connected} if there exists an ordering $w=(p_{1},\ldots,p_{d})$ satisfying the following conditions:
\begin{equation}\label{neweq}
\text{$\{p_{1},\ldots,p_{n}\}$ is a basis of $M$ and  $\lvert{(\mathcal{L}_{w}[i])_{p_{i}}}\rvert\leq 2$ for each $i$, with $n<i\leq d$.}
\end{equation}
A hyperplane arrangement is {\em inductively connected} if its associated matroid is inductively connected.
\end{definition}


\begin{example}
Consider the matroid $M$ of rank $4$, defined over the ground set $[12]$, with subspaces: 
\begin{align*}
\mathcal{L}_M=\{\{1,2,3,&4\},\{3,4,5,6\},\{5,6,7,8\},\{7,8,9,10\},\{9,10,11,12\}
,\\
&\{11,12,1,2\},\{1,5,9\},\{2,6,10\},\{3,7,11\}\}.
\end{align*}
It is straightforward to verify that the ordering $w=(1,2,3,5,6,7,9,10,11,4,8,12)$ satisfies the conditions of Equation~\eqref{neweq}, implying that $M$ is inductively connected.
\end{example}

We now characterize inductively connected matroids within the class~of~paving~matroids.

\begin{lemma}
For a paving matroid $M$, let $Q_{M}=\{p\in \mathcal{P}_{M}:\size{\mathcal{L}_{p}}>2\}$.
If $M$ is an $n$-paving matroid, then $M$ is inductively connected if and only if there is no submatroid $\emptyset \subsetneq N \subset M$ satisfying $Q_{N}=N$.
\end{lemma}

\begin{proof}
From Equation~\eqref{neweq}, we have that $M$ is inductively connected if and only if there exists an ordering $w=(p_{1},\ldots,p_{d})$ such that:
\begin{equation}\label{neweq2}
\text{$\{p_{1},\ldots,p_{n}\}$ is a basis of $M$ and $p_{i}\not \in Q_{M_{w}[i]}$ for all $n<i\leq d$.}
\end{equation}
Suppose there exists a submatroid $\emptyset \subsetneq N \subset M$ with $Q_{N}=N$. Assume $M$ is inductively connected with an ordering $w=(p_{1},\ldots,p_{d})$ as in Equation~\eqref{neweq2}. Let $i$ be the smallest index such that $p_{i}\in N$. Since $Q_{N}=N$, we have $p_{i}\in Q_{N}\subset Q_{M_{w}[i]}$, which is a contradiction.

Conversely, if no submatroid $\emptyset \subsetneq N \subset M$ satisfies $Q_{N}=N$, we can construct an ordering $w$ that fulfills the conditions of Equation~\eqref{neweq2}. This is done by recursively selecting $p_{i}$ as any point of $N_i:=M \backslash \{p_{i+1},\ldots,p_{d}\}$ such that $p_{i}\not \in Q_{N_i}$, which is always possible by our assumption. 
\end{proof}

\begin{example}
Consider the $4$-paving matroid $M$ on $[8]$ with dependent hyperplanes 
\[\mathcal{L}_M=\{\{1,2,3,4\},\{3,4,5,6\},\{5,6,7,8\},\{7,8,1,2\},\{1,2,5,6\}\}.\]
The ordering $w=(1,2,5,6,3,4,7,8)$ satisfies the conditions of Equation~\eqref{neweq}, hence $M$ is inductively connected. However, adding the circuit $\{3,4,7,8\}$ to the set of dependent hyperplanes results in a paving matroid that is not inductively connected, since all points would then have degree $3$.
\end{example}

\begin{remark}
Note that for rank $3$ matroids, the family of inductively connected matroids coincides with the class of {\em solvable} matroids introduced in \cite{Emiliano-Fatemeh5}. Furthermore, in the context of line arrangements, this family aligns with the definition of \textit{inductively connected} line arrangements from \cite{nazir2012connectivity}.
\end{remark}

\begin{definition}\normalfont \label{new definition}
Let $M$ be a matroid of rank $n$ with the elements $\mathcal{P}_M$. We fix the following notions:

\medskip
(i) For a point $p\in \mathcal{P}_M$ and $\mathcal{L}_{p}$ as given in Definition~\ref{general}, we define  \begin{equation}\label{am}
a_{p}=\sum_{l\in \mathcal{L}_{p}} \ \rank(l) -n(\size{\mathcal{L}_{p}}-1).
\end{equation}

\smallskip
(ii) A point $p\in \mathcal{P}_M$ is called $2$-\textit{simple} if either $\size{\mathcal{L}_{p}}\leq 1$ or
$\mathcal{L}_{p}=\{l_{1},l_{2}\}$ with $\rank (l_{1}\cup l_{2})=n$.

\medskip
(iii) For a point $p\in \mathcal{P}_M$ of degree two, where $\mathcal{L}_{p}=\{l_{1},l_{2}\}$, we denote 
\begin{equation}\label{bp}b_{p}=\rank(l_{1})+\rank(l_{2})-\rank(l_{1}\cup l_{2}).
\end{equation}
Note that $b_{p}$ coincides with $a_{p}=\rank(l_{1})+\rank(l_{2})-n$ if and only if $p$ is a $2$-simple point. 

\medskip
(iv) Assuming $M$ is inductively connected and $w=(p_{1},\ldots,p_{d})$ an ordering of its points as in \eqref{neweq}. For each $i\in [d]$, we define $\widetilde{\tau}_{i}$ as follows:
\[
\widetilde{\tau}_{i} = \begin{cases} 
\tau_{i} & \text{if } \lvert {(\mathcal{L}_{w}[i])}_{p_{i}}\rvert \leq 1, \\
\rank(l_{1}) + \rank(l_{2}) - \rank(l_{1} \cup l_{2}) & \text{if } {(\mathcal{L}_{w}[i])}_{p_{i}} = \{l_{1}, l_{2}\}.
\end{cases}
\]
\end{definition}

We now provide a characterization of $2$-simple points in elementary split matroids.

\begin{lemma}\label{split is simple}
Let $M$ be an 
elementary split matroid of rank $n$. Then every point in $M$ of degree at most two is $2$-simple.
\end{lemma}
\begin{proof}
Let $\mathcal{H}=\{H_{1},\ldots,H_{q}\}$ be the corresponding hypergraph of $M$. By Lemma~\ref{subspaces split}, the set of subspaces of $M$ is $\mathcal{L}=\{H_{1},\ldots,H_{q}\}$. To prove the statement, it suffices to show that $\rank(H_{i}\cup H_{j})=n$ for any $i\neq j\in [q]$. Since $\size{H_{i}\cap H_{j}}\leq r_{i}+r_{j}-n$, it follows that $\size{H_{i}\cup H_{j}}\geq n$. By Definition~\ref{elementary}, to verify that $\rank(H_{i}\cup H_{j})=n$, it remains to show that $\size{(H_{i}\cup H_{j})\backslash H_{k}}\geq n-r_{k}$ for all $k\in [q]$. Without loss of generality, assume $k\neq i$; then, since $\size{H_{i}\cap H_{k}}\leq r_{i}+r_{k}-n$, we conclude that $\size{(H_{i}\cup H_{j})\backslash H_{k}}\geq \size{H_{i}\backslash H_{k}}\geq r_{i}-\size{H_{i}\cap H_{k}}\geq n-r_{k}$, which completes the proof.
\end{proof}

\begin{lemma}\label{geno 2}
Let $M$ be an inductively connected elementary split matroid on $[d]$. 
Then, we have \[\sum_{i=1}^{d} \widetilde{\tau}_{i}=\dnaive(\Gamma_{M}).\] 
\end{lemma}

\begin{proof}
Let $w=(p_{1},\ldots,p_{d})$ be the corresponding ordering as in Equation~\eqref{neweq}. By Lemma~\ref{split is simple}, every $p_{i}$ is a $2$-simple point of $M_{w}[i]$, implying that $\widetilde{\tau_{i}}=\tau_{i}$ for all $i\in [d]$. Thus, 
$\textstyle \sum_{i=1}^{d} \widetilde{\tau}_{i}=\textstyle \sum_{i=1}^{d}\tau_{i}$. 
Now, the result holds by 
Proposition~\ref{igual 3}, as $M$ is elementary split.
\end{proof}

\subsection{Realization spaces of inductively connected matroids}\label{sec:realization}

The main result of this subsection, Theorem~\ref{gen}, establishes the irreducibility and smoothness of the realization space for any realizable inductively connected matroid, along with an explicit formula for its dimension. Moreover, we prove that this dimension coincides with the naive dimension when the matroid is both inductively connected and elementary split. Furthermore, we provide an explicit description of the realization space as an open Zariski subset of a complex space.

\medskip

In the next proposition, we examine the realization space of a matroid obtained by adding a point of degree at most two to a given matroid. (Recall the notions $a_p$ and $b_p$ from Definition~\ref{new definition}).

\begin{proposition}\label{propo a1}
Let $M$ be a matroid and let $p\in \mathcal{P}_M$ be a point of degree at most two. Define $N$ as the submatroid with points $\mathcal{P}_M\backslash \{p\}$. Then, the following hold:
\begin{itemize}
\item[{\rm (i)}] If $\deg(p)=0$, then $\Gamma_{M}$ is isomorphic to a Zariski open subset of $\Gamma_{N}\times \CC^{n}$.
\item[{\rm (ii)}] If $\deg(p)=1$, then $\Gamma_{M}$ is isomorphic to a Zariski open subset of $\Gamma_{N}\times \CC^{a_{p}}$.
\item[{\rm (iii)}] If $\deg(p)=2$, then $\Gamma_{M}$ is isomorphic to an open Zariski subset of $\Gamma_{N}\times \CC^{b_{p}}$.
\end{itemize}
\end{proposition}

\begin{proof}
Let $\Gamma_{M}\rq$ denote the set of all collections of vectors $\gamma \in V_{\mathcal{C}(M)}$ such that $\restr{\gamma}{N}\in \Gamma_{N}$. Note that $\Gamma_{M}\subset \Gamma_{M}\rq$, since $\Gamma_{M}\subset V_{\mathcal{C}(M)}$ and $\restr{\gamma}{N}\in \Gamma_{N}$ for any $\gamma\in \Gamma_{M}$. Furthermore, since $\Gamma_{M}\rq\subset V_{\mathcal{C}(M)}$ and $\Gamma_{M}=V_{\mathcal{C}(M)}\backslash V(J_{M})$, where $V(J_{M})$ is defined as in Equation~\eqref{jm}, we have
\[\Gamma_{M}\subseteq \Gamma_{M}\rq\backslash V(J_{M})\subseteq V_{\mathcal{C}(M)}\backslash V(J_{M})=\Gamma_{M},\]
implying that
$\Gamma_{M}=\Gamma_{M}\rq\backslash V(J_{M})$.
Hence, 
$\Gamma_{M}$ is a Zariski open subset of $\Gamma_{M}\rq$. Thus, it is enough to show that $\Gamma_{M}\rq$ is isomorphic to the desired set. To achieve this, we consider each case separately:

\medskip
{\bf (i)} Suppose that $\deg(p)=0$. In this case, there are no constraints on the point $p$, so any element of $\Gamma_{M}\rq$ can be derived from a realization of $N$ alongside any choice of a vector in $\CC^{n}$ associated with $p$. Based on this, we define the polynomial function 
 \begin{equation}\label{degree 0}
 \begin{aligned}
F:\ &\Gamma_{N}\times \CC^{n}\rightarrow \Gamma_{M}\rq \quad\text{with}\quad
 (\gamma,v)  \mapsto \widetilde{\gamma},
\end{aligned}
 \end{equation}
where $\widetilde{\gamma}$ is defined as $\widetilde{\gamma}_{q}=\gamma_{q}$ for $q\neq p$ and $\widetilde{\gamma}_{p}=v$. It is clear that $F$ is bijective and that both $F$ and $F^{-1}$ are polynomial maps. Therefore, we conclude that $\Gamma_{M}\rq \cong \Gamma_{N}\times \CC^{n}$.

\medskip
{\bf (ii)}  Suppose that $\deg(p)=1$ and let $\mathcal{L}_{p}=\{l\}$ with $q_{1},\ldots,q_{\rank(l)}$ forming a basis of $l$. In this case, we define the polynomial function 
 \begin{equation}\label{degree 1}
 \begin{aligned}
F: \ &\Gamma_{N}\times \CC^{\rank(l)}\rightarrow \Gamma_{M}\rq \quad\text{with}\quad
 (\gamma,c_{1},\ldots,c_{\rank(l)}) \mapsto \widetilde{\gamma},
\end{aligned}
 \end{equation}
where $\widetilde{\gamma}$ is defined as $\widetilde{\gamma}_{q}=\gamma_{q}$ for $q\neq p$ and $\widetilde{\gamma}_{p}=\textstyle \sum_{i=1}^{\rank(l)}c_{i}\gamma_{q_{i}}$. Since $\{\gamma_{q_{1}},\ldots,\gamma_{q_{\rank(l)}}\}$ forms a basis of $\gamma_{l}$, it is clear that $F$ is bijective and that both $F$ and $F^{-1}$ are polynomial maps. Therefore, we have $\Gamma_{M}\rq \cong \Gamma_{N}\times \CC^{\rank(l)}$. 
This establishes that $\Gamma_{M}$ is isomorphic to an open Zariski subset of $\Gamma_{N}\times \CC^{a_{p}}$.

\medskip
{\bf (iii)}  Suppose $\deg(p)=2$. Let $\mathcal{L}_{p}=\{l_{1},l_{2}\}$ and $m=\rank(l_{1}\cup l_{2})$. 
Choose bases 
$
\{q_{1},\ldots,q_{\rank(l_{1})}\}$ and $\{r_{1},\ldots,r_{\rank(l_{2})}\}$
of $l_{1}$ and $l_{2}$, respectively. Consider the basis 
$\mathcal{B}_{1}=\{q_{1},\ldots,q_{\rank(l_{1})},q_{1}\rq,\ldots,q_{m-\rank(l_{1})}\rq\}$ of $l_{1}\cup l_{2}$, extending the independent set $\{q_{1},\ldots,q_{\rank(l_{1})}\}$. Since the set $\{q_{1},\ldots,q_{\rank(l_{1})},r_{1},\ldots,r_{\rank(l_{2})}\}$ has full rank on $l_{1}\cup l_{2}$, we can also choose a basis $\mathcal{B}_{2}=\{r_{1},\ldots,r_{\rank(l_{2})},r_{1}\rq,\ldots,r_{m-\rank(l_{2})}\rq\}$ of $l_{1}\cup l_{2}$, ensuring that $\{r_{1}\rq,\ldots,r_{m-\rank(l_{2})}\rq\}\subset \{q_{1},\ldots,q_{\rank(l_{1})}\}$.

\medskip
\noindent {\bf Claim.} For any $\gamma \in \Gamma_{N}$, let $S=\is{\gamma_{q_{1}\rq},\ldots,\gamma_{q_{m-\rank(l_{1})}\rq},\gamma_{r_{1}\rq},\ldots,\gamma_{r_{m-\rank(l_{2})}\rq}}$. Then, we have 
\begin{equation}\label{equg}
\gamma_{l_{1}}\cap \gamma_{l_{2}}+S
=\gamma_{l_{1}}+\gamma_{l_{2}}.
\end{equation}
First note that
$\dim(\gamma_{l_{1}}\cap \gamma_{l_{2}})=\dim(\gamma_{l_{1}})+\dim(\gamma_{l_{2}})-\dim(\gamma_{l_{1}}+\gamma_{l_{2}})=\rank(l_{1})+\rank(l_{2})-\rank(l_{1}\cup l_{2})$.

To prove the claim, we need to show that $(\gamma_{l_{1}}\cap \gamma_{l_{2}})\cap S=\{0\}$. Suppose that for some $x\in \gamma_{l_{1}}\cap \gamma_{l_{2}}$, we have 
$\textstyle{x+\sum_{i}a_{i}\gamma_{q_{i}\rq}+\sum_{j}b_{j}\gamma_{r_{j}\rq}=0}$. Since $\gamma_{r_{j}\rq}\in \gamma_{l_{1}}$ for every $j$, it follows that $\textstyle \sum_{i}a_{i}\gamma_{q_{i}\rq}\in \gamma_{l_{1}}$. Given that $\mathcal{B}_{1}$ is a basis, $\gamma_{l_{1}}\cap \is{\gamma_{q_{1}\rq},\ldots,\gamma_{q_{m-\rank(l_{1})}\rq}}=\{0\}$, implying $a_{i}=0$ for all $i$. Thus,~\eqref{equg} reduces to 
$\textstyle{x+\sum_{j}b_{j}\gamma_{r_{j}\rq}=0}$.
Since $\mathcal{B}_{2}$ is a basis, 
$\is{\gamma_{r_{1}\rq},\ldots,\gamma_{r_{m-\rank(l_{2})}\rq}}\cap \gamma_{l_{2}}=\{0\}$. Hence, as $x\in \gamma_{l_{2}}$, we find that $b_{j}=0$ for all $j$.  This shows that all the coefficients in~\eqref{equg} must vanish, proving the claim. \qed

\medskip
Our goal now is to derive an expression for a basis of $\gamma_{l_{1}}\cap \gamma_{l_{2}}$ as polynomial functions on the vectors of $\gamma$ for any $\gamma\in \Gamma_{N}$.
We fix $\mathcal{B}=\{p_{1},\ldots,p_{m}\}$ a basis of $l_{1}\cup l_{2}$. For each $\gamma \in \Gamma_{N}$, let $A_{\gamma}$ be the $n\times m$ matrix whose columns are the vectors $\{\gamma_{p_{i}}:i\in [m]\}$. Notice that $A_{\gamma}$ has full column rank, with its columns forming a basis of $\gamma_{l_{1}}+ \gamma_{l_{2}}$. For each $q\in \closure{l_{1}\cup l_{2}}$, we denote by $\gamma_{q}\rq \in \CC^{m}$ the vector $(A_{\gamma}^{\text{T}}A_{\gamma})^{-1}A_{\gamma}^{\text{T}}\gamma_{q}$.
The vector $\gamma_{q}\rq$ represents the coordinates of $\gamma_{q}$ in the basis $\{\gamma_{p_{i}}:i\in [m]\}$, meaning $\gamma_{q}=A_{\gamma}\gamma_{q}\rq$. For ease of notation, we denote $B_{\gamma}=(A_{\gamma}^{\text{T}}A_{\gamma})^{-1}A_{\gamma}^{\text{T}}$.

\medskip
Let $\{s_{1},\ldots,s_{b_{p}}\}$ be the elements of $\mathcal{B}_{1}\backslash \{q_{1}\rq,\ldots,q_{m-\rank(l_{1})}\rq,r_{1}\rq,\ldots,r_{m-\rank(l_{2})}\rq\}$. Since $\mathcal{B}_{1}$ is a basis,  
$S+\is{\gamma_{s_{1}},\ldots,\gamma_{s_{b_{p}}}}=\gamma_{l_{1}}+\gamma_{l_{2}}$,
which implies that $B_{\gamma}(S)+\is{\gamma_{s_{1}}\rq,\ldots,\gamma_{s_{b_{p}}}\rq}=\CC^{m}$. Applying Lemma~\ref{basis} with $B_{\gamma}(\gamma_{l_{1}}\cap \gamma_{l_{2}}),B_{\gamma}(S)$ and $\{\gamma_{s_{1}}\rq,\ldots,\gamma_{s_{b_{p}}}\rq\}$, we find that the set 
\[\{v_{i}:=\gamma_{q_{1}}\rq\cdots \gamma_{q_{\rank(l_{1})}}\rq\wedge \gamma_{r_{1}}\rq\cdots \gamma_{r_{\rank(l_{2})}}\rq\wedge (\gamma_{q_{1}\rq}\rq\cdots\gamma_{q_{m-\rank(l_{1})}\rq}\rq\vee \gamma_{r_{1}\rq}\rq\cdots \gamma_{r_{m-\rank(l_{2})}\rq}\rq \vee \gamma_{s_{i}}\rq):i\in [b_{p}]\}  \]
forms a basis of $B_{\gamma}(\gamma_{l_{1}}\cap \gamma_{l_{2}})$. Thus, the set of vectors 
$\{A_{\gamma}v_{i}:i\in [b_{p}]\}$
forms a basis of $\gamma_{l_{1}}\cap \gamma_{l_{2}}$. Therefore, there exist polynomial functions $g_{1},\ldots,g_{b_{p}}$ such that $\{g_{1}(\gamma),\ldots,g_{b_{p}}(\gamma)\}$ forms a basis of $\gamma_{l_{1}}\cap \gamma_{l_{2}}$ for each $\gamma \in \Gamma_{N}$. We can then define the polynomial function
 \begin{equation}\label{degree 2}
 \begin{aligned}
F: \ &\Gamma_{N}\times \CC^{b_{p}}\rightarrow \Gamma_{M}\rq \quad\text{with}\quad
(\gamma,c_{1},\ldots,c_{b_{p}}) \mapsto \widetilde{\gamma},
\end{aligned}
 \end{equation}
where $\widetilde{\gamma}$ is defined as $\widetilde{\gamma}_{q}=\gamma_{q}$ for $q\neq p$ and $\widetilde{\gamma}_{p}=\textstyle \sum_{i=1}^{b_{p}}c_{i}g_{i}(\gamma)$. It follows that $F$ is a bijective polynomial map, and similarly, we can verify that $F^{-1}$ is a polynomial map obtained by taking the inverse matrix corresponding to the basis $\{g_{1}(\gamma),\ldots,g_{b_{p}}(\gamma)\}$ of $\gamma_{l_{1}}\cap \gamma_{l_{2}} $. Hence, we have established that $\Gamma_{M}\rq$ is isomorphic to $\Gamma_{N}\times \CC^{b_{p}}$, as desired.
\end{proof}

\begin{remark}\label{smo}
\begin{itemize}
   \item The proof of Proposition~\ref{propo a1} also holds for the matroid stratum $\text{Gr}(M,\CC)$. In particular, smoothness is preserved by adding points of degree at most two. 
 \item Proposition~\ref{propo a1}(iii) can be regarded as a generalization of \textup{\cite[Proposition~3.5]{corey2023singular}}, which establishes a similar result for points lying in at most two dependent hyperplanes. 
\item In Proposition~\ref{propo a1}(iii), if $a_{p}=1$ and $p$ is $2$-simple, then the vector in the intersection $\gamma_{l_{1}}\cap \gamma_{l_{2}}$ can be directly expressed as $\gamma_{l_{1}}\wedge \gamma_{l_{2}}$.
\end{itemize}
\end{remark}

Proposition~\ref{propo a1} and Theorem~\ref{gen} imply that the realization spaces of \textit{inductively connected} matroids are irreducible, extending \textup{\cite[Proposition~3.3]{nazir2012connectivity}}, which showed this property for inductively connected line arrangements, or equivalently, for rank $3$ matroids, as discussed in Example~\ref{corresp}.

\begin{theorem}\label{gen}
Let $M$ be an inductively connected matroid on $[d]$. 
Then $\Gamma_{M}$ is isomorphic to a Zariski open subset of a complex space of dimension $\textstyle \sum_{i=1}^{d} \widetilde{\tau}_{i}.$ 
\end{theorem}

\begin{proof}
We prove this statement by induction on $d$. When $d = n$, the result is trivially true, confirming the base case of the induction. Assume that the result holds for any matroid with at most $d - 1$ points. Let $M$ be a matroid with $d$ points, and let $w = (p_{1}, \ldots, p_{d})$ be an ordering as in Equation~\eqref{neweq}. By the induction hypothesis, $\Gamma_{M_{w}[d - 1]}$ is isomorphic to a Zariski open subset of a complex space of dimension $\textstyle \sum_{i = 1}^{d - 1} \widetilde{\tau}_{i}$. 
Applying Proposition~\ref{propo a1}, depending on whether $p_{d}$ has degree zero, one, or two, we deduce that $\Gamma_{M}$ is isomorphic to a Zariski open subset of a complex space of dimension $\textstyle \sum_{i = 1}^{d} \widetilde{\tau}_{i}$, as desired.
\end{proof}

By applying Theorem~\ref{gen} to the families of inductively connected elementary split matroids, and inductively connected matroids, we obtain the following corollaries.

\begin{corollary}\label{gen 2}
Let $M$ be an inductively connected elementary split matroid. Then $\Gamma_{M}$ is isomorphic to a Zariski open subset of a complex space of dimension $\dnaive(\Gamma_{M})$. 
\end{corollary}

\begin{proof}
Let $w=(p_{1},\ldots,p_{d})$ be an ordering of $M$ as in Equation~\eqref{neweq}. By Theorem~\ref{gen}, $\Gamma_{M}$ is isomorphic to a Zariski open subset of a complex space of dimension $\textstyle \sum_{i=1}^{d} \widetilde{\tau}_{i}$. Then, the result follows from applying Lemma~\ref{geno 2}.
\end{proof}

\begin{corollary}\label{coro pa}
Let $M$ be an inductively connected realizable matroid. Then $\Gamma_{M}$ is a smooth and irreducible variety. 
Moreover, if two hyperplane arrangements $\mathcal{A}_{1}$ and  $\mathcal{A}_{2}$ have $M$ as their associated matroid, then we have $N(\mathcal{A}_{1})\cong N(\mathcal{A}_{2})$. 
\end{corollary}

\begin{proof}
The first statement follows from Theorem~\ref{gen}, while the second stems from Corollary~\ref{coro ar}.
\end{proof}

Next, we present examples of matroids satisfying the realizability condition in Corollary~\ref{coro pa}.

\begin{example}\label{ex:real} A matroid is called {\em nilpotent} if it contains no submatroid in which all points have degree at least two. If $M$ is a nilpotent matroid, or all its points are $2$-simple, then, $M$ is realizable; see Theorem~4.11 and Corollary~5.26 in \cite{Emiliano-Fatemeh5}. 
\end{example} 

Another specific family of realizable matroids consists of paving matroids without points of degree greater than two. A complete characterization of their realization spaces is provided in \cite{Emiliano-Fatemeh4}.

\begin{example}
Let $M$ be a paving matroid without points of degree greater than two. In \cite{Emiliano-Fatemeh4}, it is shown that 
$M$ is realizable. Furthermore, by Corollaries~\ref{gen 2} and~\ref{coro pa}, $\Gamma_{M}$ is a smooth and irreducible variety of dimension $\dnaive(\Gamma_{M})$. Applying Lemma~\ref{na pav}, we derive an explicit formula for this dimension. Moreover, a complete set of defining equations for the matroid variety $V_{M}$ is provided in \cite{Emiliano-Fatemeh4}. 
\end{example}

\subsection{Computing defining equations of realization spaces}\label{subsec:ex}
Here, we will first provide an explicit construction for computing the defining polynomials of the Zariski open subset, as outlined in Theorem~\ref{gen}. Then, we provide detailed examples of this construction.

\begin{procedure}\label{procedure}
\normalfont
Let $M$ be an inductively connected matroid of rank $n$ on $[d]$, and let $w=(p_{1},\ldots,p_{d})$ be an ordering of its elements as in Equation~\eqref{neweq}.
Note that for each $i\geq n$, the matroid $M_{w}[i]$ is also inductively connected; thus, by Theorem~\ref{gen}, $\Gamma_{M_{w}[i]}$ is isomorphic to a Zariski open subset of a complex space.
We will recursively construct the defining polynomials of these open subsets for $i\geq n$.

\medskip
\noindent{\bf First step:} We have $\Gamma_{M_{w}[n]}\cong \textup{Gl}_{n}(\CC)$, the space of $n\times n$ matrices with non-zero determinant. 

\medskip
\noindent{\bf Recursive step:}  
Let $i\geq n+1$. 
Assume we have an explicit description of the polynomials defining $\Gamma_{M_{w}[i-1]}$. Using polynomial maps in~\eqref{degree 0},~\eqref{degree 1} and~\eqref{degree 2}, depending on the degree of $p_{i}$ in $M_{w}[i]$, we can explicitly describe the realization space $\Gamma_{M_{w}[i]}$ from the known description of $\Gamma_{M_{w}[i-1]}$.
\end{procedure}

\begin{remark}
As noted in Remark~\ref{smo}, the argument from Procedure~\ref{procedure} can be also applied to the matroid stratum $\text{Gr}(M,\CC)$, allowing us to obtain its explicit description.
\end{remark}

We now provide two examples in which we apply Procedure~\ref{procedure} to explicitly describe this parametrization, focusing specifically on the space $\text{Gr}(M,\CC)$ for simplicity.

\begin{example}
Let $M$ be the $4$-paving matroid on $[8]$ with dependent hyperplanes:
\[\mathcal{L}_M=\{\{1,2,3,4\},\{3,4,5,6\},\{5,6,7,8\},\{1,3,5,7\}\}.\]
The ordering $w=(1,2,3,5,4,6,7,8)$ satisfies the conditions of Equation~\eqref{neweq}, confirming that $M$ is inductively connected. 
Therefore, we can represent the matrix of indeterminates in the following form:

\[ \begin{pmatrix}
1 & 0 & 0 & 0 & x_{1} & x_{5} & x_{9} & x_{13}\\
0 & 1 & 0 & 0 & x_{2} & x_{6} & x_{10} & x_{14}\\
0 & 0 & 1 & 0 & x_{3} & x_{7}& x_{11}& x_{15}\\
0 & 0 & 0 & 1 & x_{4}& x_{8}& x_{12}& x_{16}
\end{pmatrix},\]
where the columns represent the points $(1,2,3,5,4,6,7,8)$. We proceed as outlined in Procedure~\ref{procedure}. 
\begin{itemize}
\item Since the point $4$ only belongs to the dependent hyperplane $\{1,2,3\}$ in $M_{w}[5]$, we have $x_{4}=0$. 
\item Since the point $6$ only belongs to the dependent hyperplane $\{3,4,5\}$ in $M_{w}[6]$, the column $6$ is of the form $(y_{1}x_{1},y_{1}x_{2},y_{2},y_{3})$. 
\item Since the point $7$ only belongs to the dependent hyperplane $\{1,3,5\}$ in $M_{w}[7]$, the column $7$ is of the form $(z_{1},0,z_{2},z_{3})$.
\item Finally, since the point $8$ only  belongs to the dependent hyperplane $\{5,6,7\}$ in $M$, the column $8$ is of the form $(w_{1}y_{1}x_{1}+w_{2}z_{1},w_{1}y_{1}x_{2},w_{1}y_{2}+w_{2}z_{2},w_{3}).$
\end{itemize}
Consequently, we can explicitly describe $\text{Gr}(M,\CC)$ as the Zariski open subset of 
$\CC[x_{1},x_{2},x_{3},y_{1},y_{2},y_{3},\\ z_{1},z_{2},z_{3},w_{1},w_{2},w_{3}]$
defined by the condition that none of the $4$-minors of the following matrix, which corresponds to a basis, vanish:
\[ \begin{pmatrix}
1 & 0 & 0 & 0 & x_{1} & y_{1}x_{2} & z_{1} & w_{1}y_{1}x_{1}+w_{2}z_{1} \\
0 & 1 & 0 & 0 & x_{2} & y_{1}x_{2} & 0 & w_{1}y_{1}x_{2}\\
0 & 0 & 1 & 0 & x_{3} & y_{2}& z_{2}&w_{1}y_{2}+w_{2}z_{2} \\
0 & 0 & 0 & 1 & 0& y_{3}& z_{3}& w_{3}
\end{pmatrix}.\]
\end{example}

\begin{example}
Let $M$ be the matroid of rank $6$ on $[12]$ with the set of subspaces 
\[\mathcal{L}_M=\{\{1,2,3,4,5,6\},\{7,8,9,10,11,12\},\{1,2,7,8\},\{3,4,9,10\},\{5,6,11,12\}\}.\]
In this case, $\mathcal{L}_M$ consists of circuits of size at most $6$. 
The ordering $w=(7,8,9,11,12,6,3,2,5,1,4,10)$ 
satisfies the conditions of Equation~\eqref{neweq}, confirming that $M$ is inductively connected. 
Therefore, we can represent the matrix of indeterminates in the following form:
\[
\left(
\begin{array}{cccccccccccc}
1&0&0 & 0 & 0 & 0 & x_{1} & x_{7} & x_{13} & x_{19}&x_{25}&x_{31}\\
0 & 1 & 0 & 0 & 0 & 0 & x_{2} & x_{8} & x_{14} & x_{20} & x_{26} & x_{32} \\
0 & 0 & 1 & 0 & 0 & 0 & x_{3} & x_{9} & x_{15} & x_{21} & x_{27} & x_{33}
 \\
0 & 0 & 0 & 1 & 0 & 0 & x_{4} & x_{10} & x_{16} & x_{22} & x_{28} & x_{34} \\
0 & 0 & 0 & 0 & 1 & 0 & x_{5} & x_{11} & x_{17} & x_{23} & x_{29} & x_{35} \\
0 & 0 & 0 & 0 & 0 & 1 & x_{6} & x_{12} & x_{18} & x_{24} & x_{30} & x_{36}
\end{array}
\right)
\]
where the columns correspond to the points $(7,8,9,11,12,6,3,2,5,1,4,10)$, respectively. We proceed the following steps as outlined in Procedure~\ref{procedure}.
\begin{itemize}
\item Observe that for $i\in \{7,8\}$, the point $p_{i}$ has degree zero on $M_{w}[i]$, implying that no condition is imposed on columns $7$ and $8$.
\item Since $\{6,11,12,5\}$ is the unique subspace of $M_{w}[9]$ containing the point $5$, we conclude that the column $9$ is dependent with the columns $4,5$ and $6$, hence expressed of form: 
$(0,0,0,y_{1},y_{2},y_{3})$. 

\item Since $\{1,2,7,8\}$ is the unique subspace of $M_{w}[10]$ containing the point $1$, column $10$ is dependent on columns $1,2$ and $8$, implying that the column $10$ is of form: 
$(z_{1},z_{2},z_{3}x_{9},z_{3}x_{10},z_{3}x_{11},z_{3}x_{12})$.
 
\item Since $\{1,2,3,4,5,6\}$ is the unique subspace of $M_{w}[11]$ containing the point $4$, the column $11$ is dependent on columns $6,7,8,9,10$. Consequently, column $11$ can be expressed as: 
\begin{equation}\label{11}
\begin{aligned}
&w_{1}\irow{x_{1},&x_{2},&x_{3},&x_{4},&x_{5},&0}+w_{2}\irow{x_{7},&x_{8},&x_{9},&x_{10},&x_{11},&0}+\\
&w_{3}\irow{0,&0,&0,&y_{1},&y_{2},&0}+w_{4}\irow{z_{1},&z_{2},&z_{3}x_{9},&z_{3}x_{10},&z_{3}x_{11},&0}+w_{5}\irow{0,&0,&0,&0,&0,&1}.
\end{aligned}
\end{equation}
\item We now turn to column $12$, representing point 
$10$, and  proceed as in the proof of Proposition~\ref{propo a1}.
\end{itemize}
{\bf Step 1.} The two subspaces containing point $10$ are $\{l_{1},l_{2}\}=\{\{7,8,9,10,11,12\},\{3,4,9,10\}\}$. We first fix bases for $l_{1}$ and $l_{2}$ that do not contain the point $10$. In this case, the bases are $\{7,8,9,11,12\}$ for $l_{1}$ and $\{3,4,9\}$ for $l_{2}$. We then extend the first basis to form $\mathcal{B}_{1}=\{7,8,9,11,12,6\}$, a basis of $M$. Next, we extend the independent set $\{3,4,9\}$ to obtain $\mathcal{B}_{2}=\{3,4,9,8,11,12\}$, another basis of $M$.

\medskip
\noindent{\bf Step 2.} Following the reasoning in the proof of Proposition~\ref{propo a1}(iii), we obtain that the intersection of the subspaces corresponding to $l_{1}$ and $l_{2}$ has a basis formed by the two vectors of indeterminates: 
\begin{equation}\label{equ gf}
\begin{aligned}
&(3\vee 4\vee 9)\wedge (7\vee 8\vee 9\vee 11\vee 12) \wedge (6\vee7\vee8\vee11\vee12),\\
&(3\vee4\vee 9)\wedge (7\vee8\vee 9\vee 11\vee 12) \wedge (6\vee 8\vee 9 \vee 11\vee 12),
\end{aligned}
\end{equation}
where the numbers in these expressions denote the vectors of indeterminates associated 
which each element of the matroid. Using that the wedge product of the Grassmann-Cayley algebra corresponds with taking the intersection of subspaces, we rewrite these vectors as: 
\begin{equation}\label{equ gf2}
\begin{aligned}
&v=(3\vee 4\vee 9)\wedge (7\vee 8\vee 11\vee 12) \quad\text{and}\quad
w=(3\vee4\vee 9)\wedge (8\vee 9\vee 11\vee 12).
\end{aligned}
\end{equation}
Note that the points $7,8,9,11,12$ correspond with the columns $1,2,3,4,5$, which correspond to the vectors $e_{1},\ldots,e_{5}$, respectively. It is straightforward to verify that $w=e_{3}$. For the vector $v$, using the expression of the wedge product of Equation~\eqref{equ wedge}, it follows that 
\begin{equation*}
\begin{aligned}
v&=(7\vee 8\vee 11\vee 12) \wedge (3\vee 4\vee 9)=(e_{1}\vee e_{2}\vee e_{4}\vee e_{5}) \wedge (3\vee 4\vee e_{3})\\
&=[e_{1},e_{2},e_{4},e_{5},3,4]e_{3}+[e_{1},e_{2},e_{3},e_{4},e_{5},3]4-[e_{1},e_{2},e_{3},e_{4},e_{5},4]3.
\end{aligned}
\end{equation*}
Thus, we conclude that the intersection of the subspaces corresponding to $l_{1}$ and $l_{2}$ has a basis composed of vectors of indeterminates:
\[z=[e_{1},e_{2},e_{3},e_{4},e_{5},3]4-[e_{1},e_{2},e_{3},e_{4},e_{5},4]3 \quad \text{and}  \quad e_{3}.\]
Replacing $3$ and $4$ with the columns $7$ and $11$ from the matrix of indeterminates, using the expression for the $11^{\text{th}}$ column from Equation~\eqref{11}, we arrive at the expression for $z$: 
\begin{equation}
\begin{aligned}
z&=x_{6}w_{1}\irow{x_{1},&x_{2},&x_{3},&x_{4},&x_{5},&0}+x_{6}w_{2}\irow{x_{7},&x_{8},&x_{9},&x_{10},&x_{11},&0}+x_{6}w_{3}\irow{0,&0,&0,&y_{1},&y_{2},&0}\\
&+x_{6}w_{4}\irow{z_{1},&z_{2},&z_{3}x_{9},&z_{3}x_{10},&z_{3}x_{11},&0}+x_{6}w_{5}\irow{0,&0,&0,&0,&0,&1}-w_{5}\irow{x_{1},&x_{2},&x_{3},&x_{4},&x_{5},&x_{6}}.
\end{aligned}
\end{equation}
Thus, we conclude that the $12^{\text{th}}$ column 
is of form $u_{1}z+u_{2}e_{3}$.
Hence, we explicitly describe $\text{Gr}(M,\CC)$ as the Zariski open subset of $\CC[x_{1},\ldots,x_{12},y_{1},y_{2},y_{3},z_{1},z_{2},z_{3},w_{1},\ldots,w_{5},u_{1},u_{2}]$ 
defined by the condition that none of the $6$-minors of the following matrix, which  corresponds to a basis element, vanish:
\[
\small{\left(
\begin{array}{cccccccccccc}
1&0&0 & 0 & 0 & 0 & x_{1} & x_{7} & 0 & z_{1}& {\scriptstyle w_{1}x_{1}+w_{2}x_{7}+w_{4}z_{1}}& {\scriptstyle u_{1}(x_{6}(w_{1}x_{1}+w_{2}x_{7}+w_{4}z_{1})-w_{5}x_{1}) }\\
0 & 1 & 0 & 0 & 0 & 0 & x_{2} & x_{8} & 0 & z_{2} & {\scriptstyle w_{1}x_{2}+w_{2}x_{8}+w_{4}z_{2} }  & {\scriptstyle u_{1}(x_{6}(w_{1}x_{2}+w_{2}x_{8}+w_{4}z_{2})-w_{5}x_{2}) } \\
0 & 0 & 1 & 0 & 0 & 0 & x_{3} & x_{9} & 0 & z_{3}x_{9} &{\scriptstyle  w_{1}x_{3}+w_{2}x_{9}+w_{4}z_{3}x_{9}}  & u_{2}
 \\
0 & 0 & 0 & 1 & 0 & 0 & x_{4} & x_{10} & y_{1} & z_{3}x_{10} & {\scriptstyle w_{1}x_{4}+w_{2}x_{10}+w_{3}y_{1}+w_{4}z_{3}x_{10}} & {\scriptstyle u_{1}(x_{6}(w_{1}x_{4}+w_{2}x_{10}+w_{3}y_{1}+w_{4}z_{3}x_{10})-w_{5}x_{4}) } \\
0 & 0 & 0 & 0 & 1 & 0 & x_{5} & x_{11} & y_{2} & z_{3}x_{11} & {\scriptstyle  w_{1}x_{5}+w_{2}x_{11}+w_{3}y_{2}+w_{4}z_{3}x_{11}} & {\scriptstyle  u_{1}(x_{6}(w_{1}x_{5}+w_{2}x_{11}+w_{3}y_{2}+w_{4}z_{3}x_{11})-w_{5}x_{5})} \\
0 & 0 & 0 & 0 & 0 & 1 & x_{6} & x_{12} & y_{3} & z_{3}x_{12} & w_{5} & 0
\end{array}
\right).
}\]
\end{example}

\subsection{Characterization of rigidity}\label{subsec:rigid}
The study of rigid matroids, those with a unique realization up to projective transformations, focuses on characterizing and identifying families of matroids with connected moduli spaces; see \cite{guerville2023connectivity}. Moreover, projective uniqueness has been explored in various contexts (see \cite{brandt2019slack, ziegler1990matroid, wenzel1991projective, miller2003unique, gouveia2020projectively}). In this subsection, we analyze the family of inductively connected matroids, offering a comprehensive characterization of their rigidity and local rigidity. We also provide a sufficient condition for inductively rigid property and characterize it for matroids that are both inductively connected and elementary split.

\medskip
Throughout this subsection, let  $M$ denote a matroid of rank $n$ on the ground set $[d]$.

\begin{definition}\normalfont
The  projective realization space of the matroid $M$ consists of all collections of points $\gamma=\{\gamma_{1},\ldots,\gamma_{d}\}\subset \PP^{n-1}$ that satisfy the following condition:
\begin{equation*}
\{\gamma_{i_{1}},\ldots,\gamma_{i_{k}}\}\ \text{are linearly dependent} \Longleftrightarrow \{i_{1},\ldots,i_{k}\} \ \text{is dependent in $M$}.
\end{equation*}
The {\em moduli space} $\mathcal{M}(M)$ of the matroid $M$ is then the quotient of this realization space by the action of the projective general linear group $\text{PGL}_{n}(\CC)=\text{GL}_{n}(\CC)/\CC^{\ast}$.
\end{definition}

\begin{remark}\label{dim}
If $M$ is realizable and contains an $n+1$ circuit, then 
$$\dim(\mathcal{M}(M))=\dim(\Gamma_{M})-(n^{2}-1+d).
$$
From now on, we will assume this hypothesis.
\end{remark}

We now present the classes of matroids known as {\em rigid} and {\em locally rigid}. 

\begin{definition}\normalfont \label{defi rigid}
Let $M$ be a realizable matroid. Then $M$ is called {\em rigid} if $\size{\mathcal{M}(M)}=1$, meaning there is a unique realization of $M$ up to projective transformation.
It is called {\em locally rigid} if $\mathcal{M}(M)$ is finite, or equivalently, if $\dim(\mathcal{M}(M))=0$.
\end{definition}

We now present a characterization of {\em rigid} and {\em locally rigid} inductively connected matroids.

\begin{proposition}\label{propo rig}
Let $M$ be a realizable inductively connected matroid on $[d]$ with $w=(p_{1},\ldots,p_{d})$ an ordering of its elements as in Equation~\eqref{neweq}. The following statements are equivalent:
\begin{itemize}
\item[{\rm (i)}] $M$ is rigid. 
\quad{\rm (ii)} $M$ is locally rigid.
\quad {\rm (iii)} $\textstyle \sum_{i=1}^{d}\widetilde{\tau}_{i}=n^{2}-1+d.$
\end{itemize}
Moreover, if $M$ is elementary split, then condition~\textup{(iii)} is equivalent to:
\begin{itemize}
   \item[{\rm (iv)}] $\dnaive(\Gamma_{M})=n^{2}-1+d.$
\end{itemize}
\end{proposition}

\begin{proof}
First, assume that $M$ is locally rigid. By definition,  $\dim(\mathcal{M}(M))=0$. By Theorem~\ref{gen}, \[\textstyle \sum_{i=1}^{d}\widetilde{\tau}_{i}=\dim(\Gamma_{M})=\dim(\mathcal{M}(M))+n^{2}-1+d=n^{2}-1+d.\]
Conversely, suppose that $\textstyle \sum_{i=1}^{d}\widetilde{\tau}_{i}=n^{2}-1+d$. By Theorem~\ref{gen}, we deduce that $\dim(\mathcal{M}(M))=0$, meaning $\mathcal{M}(M)$ is finite. Since Theorem~\ref{gen} also ensures that $\mathcal{M}(M)$ is connected, we conclude $\lvert \mathcal{M}(M)\rvert =1$, implying that $M$ is rigid. Now, assume that $M$ is both inductively connected and elementary split. In this case, the equivalence of (iii) and (iv) follows directly from Lemma~\ref{geno 2}.
\end{proof}

\begin{example}\label{qs}
Consider the quadrilateral configuration $\text{QS}$ depicted in Figure~\ref{fig:combined} (Right), which is the paving matroid of rank $3$ on the ground set $\{p_{1},\ldots,p_{6}\}$ with the following set of lines: \[\mathcal{L}_{QS}=\{\{p_1,p_2,p_3\},\{p_1,p_5,p_{6}\},\{p_3,p_4,p_5\},\{p_2,p_4,p_6\}\}.\]
Since all points have degree $2$, $QS$ is inductively connected. Applying Equation~\eqref{naive pav}, we have that $\dnaive(\Gamma_{QS})=14$. Consequently, Proposition~\ref{propo rig} confirms that the matroid $\text{QS}$ is rigid.
\end{example}

We now define {\em inductively rigid} 
matroids,  originally introduced in \cite{guerville2023connectivity} for line arrangements.

\begin{definition}\normalfont\label{inductive}
A matroid $M$ is called \textit{inductively rigid} if there exists an ordering $w=(p_{1},\ldots,p_{d})$ of its elements such that $M_{w}[i]$ is rigid for all $i\in [d]$ and $\{p_{1},\ldots,p_{n+1}\}$ forms a circuit of $M$. 
\end{definition}

We present a sufficient condition for a rigid inductively connected matroid to be inductively rigid.

\begin{proposition}\label{proposition rigid 2}
Let $M$ be a rigid, inductively connected matroid of rank $n$, with an ordering $w=(p_{1},\ldots,p_{d})$ 
such that $\{p_{1},\ldots,p_{n+1}\}$ forms a circuit of $M$. Then $M$ is inductively rigid.
\end{proposition}

\begin{proof}
By Theorem~\ref{gen} and by the rigidity of $M$, we have the following equality:
\begin{equation}\label{rig tau}
n^{2}-1+d=\dim( \Gamma_{M})=\textstyle \sum_{i=1}^{d}\widetilde{\tau}_{i}=n^{2}+n+\sum_{i=n+2}^{d}\widetilde{\tau}_{i}.
\end{equation}
Since the numbers $\widetilde{\tau}_{i}$ are strictly positive, \eqref{rig tau} implies that $\widetilde{\tau}_{i}=1$ for all $i\geq n+2$. Moreover, as $M_{w}[i]$ is inductively connected for each $i\geq n+2$, Theorem~\ref{gen} yields 
$\textstyle{\dim( \Gamma_{M_{w}[i]})=\sum_{j=1}^{i}\widetilde{\tau}_{j}=n^{2}+i-1}.$
Thus, by Proposition~\ref{propo rig}, it follows that $M_{w}[i]$ is rigid for all $i\geq n+2$. Moreover, since $\{p_{1},\ldots,p_{n+1}\}$ forms a circuit, $M_{w}[i]$ is also rigid for all $i\in [n+1]$. Thus, $M$ is inductively rigid.
\end{proof}

We now present a characterization of inductively rigid matroids, when they are both inductively connected and elementary split.

\begin{proposition}\label{prop loc}
Let $M$ be a realizable inductively connected elementary split matroid. Then, $M$ is inductively rigid if and only if there exists an ordering $w=(p_{1},\ldots,p_{d})$ of its elements such that:
$\{p_{1},\ldots,p_{n+1}\}$ forms a circuit, and 
$\tau_{i}=1$ for all $i\geq n+2$.
\end{proposition}

\begin{proof}
Assume there exists an ordering $w$ satisfying the specified conditions. Since $\{p_{1},\ldots,p_{n+1}\}$ forms a circuit, it follows that $M_{w}[i]$ is rigid for all $i\in [n+1]$. For each $i\geq n+2$, we have $\dnaive(\Gamma_{M_{w}[i]})=n^{2}+i-1$, 
by Proposition~\ref{igual 3}. Since the submatroid $M_{w}[i]$ is inductively connected and elementary simple, Proposition~\ref{propo rig} implies that $M_{w}[i]$ is rigid. Therefore, $M$ is inductively rigid.

Now, suppose $M$ is inductively rigid,
and $w=(p_{1},\ldots,p_{d})$ is an ordering as in Definition~\ref{inductive}. For each $i\geq n+1$, 
$M_{w}[i]$ is inductively connected, elementary split, and rigid. Then, by Proposition~\ref{propo rig}, it follows that $\dnaive(\Gamma_{M_{w}[i]})=n^{2}+i-1$ for every $i\geq n+1$, and by Proposition~\ref{igual 3},
$\textstyle{\sum_{j=1}^{i}\tau_{j}=n^{2}+i-1}$. 
This implies that $\tau_{i}=1$ for all $i\geq n+2$. Hence, $w$ satisfies the desired conditions.
\end{proof}

\begin{example}
Consider the matroid $M$ of rank $4$ defined on the ground set $[6]$ with the set of subspaces $\mathcal{L}_M=\{\{1,2,3,6\},\{4,5,6\}\}$, which, in this case, coincides with the set of circuits of size at most $4$.
The matroid $M$ is elementary split. The ordering $w=(1,2,3,4,5,6)$ satisfies the conditions in~\eqref{neweq}, confirming that $M$ is inductively connected. We have that $\tau(M,w)=(4,4,4,4,4,1)$. 
Since 
$\{1,2,3,4,5\}$ forms a circuit and $\tau_{6}=1$, it follows from Proposition~\ref{prop loc} that $M$ is inductively rigid.
\end{example}

\begin{example}
Consider the configuration $\text{QS}$ from Example~\ref{qs}. This matroid is elementary split since it is a paving matroid. For the ordering $w=(p_2,p_3,p_5,p_6,p_1,p_4)$, we have that $\tau(M,w)=(3,3,3,3,1,1)$. By applying Proposition~\ref{prop loc}, we conclude that $\text{QS}$ is inductively rigid.
\end{example}

\printbibliography

\end{document}